\documentclass[a4paper,12pt]{article}
\usepackage[latin1]{inputenc}
\usepackage{amsthm}
\usepackage{amsfonts}
\usepackage[dvips]{graphicx}
\usepackage{fancyhdr}
\usepackage{latexsym}
\usepackage{amsmath}
\usepackage{color}
\usepackage{stmaryrd}

\newtheorem{lem}{Lemma}
\newtheorem{thm}{Theorem}
\newtheorem{defn}{Definition}

\newtheorem{cor}{Corollary}
\newtheorem{oss}{Remark}

\newenvironment{bettirev}{\color{black}}{\color{black}}
\newcommand{\bber}{\begin{bettirev}}
\newcommand{\eber}{\end{bettirev}}

\allowdisplaybreaks[4]

\numberwithin{equation}{section}

\begin{document}
\title{On a diffuse interface model of tumor growth}

\author{
{Sergio Frigeri \thanks{{\color{black} Weierstrass Institute for Applied Analysis and Stochastics, Mohrenstr. 39, D-10117 Berlin,
Germany.
E-mail: \textit{SergioPietro.Frigeri@wias-berlin.de}  The author
is supported by the FP7-IDEAS-ERC-StG
Grant \#256872 (EntroPhase)}}}
\and
{Maurizio Grasselli\thanks{Dipartimento di Matematica,
Politecnico di Milano,
Milano I-20133, Italy.
E-mail: \textit{mau\-ri\-zio.grasselli@polimi.it}}}
\and
{Elisabetta Rocca\thanks{{\color{black}
Weierstrass Institute for Applied Analysis and Stochastics, Mohrenstr. 39, D-10117 Berlin,
Germany.
E-mail: \textit{Elisabetta.Rocca@wias-berlin.de} and Dipartimento di Matematica ``F. Enriques'',
Universit\`{a} degli Studi di Milano, Milano I-20133, Italy.
E-mail: \textit{elisabetta.rocca@unimi.it} The author
is supported by the FP7-IDEAS-ERC-StG
Grant \#256872 (EntroPhase)}}}}

\maketitle


\begin{abstract}\noindent
We consider a diffuse interface model of tumor growth proposed by A.~Hawkins-Daruud et al.
This model consists of the Cahn-Hilliard equation for the tumor cell fraction $\varphi$ nonlinearly coupled with
a reaction-diffusion equation for $\psi$, which represents the nutrient-rich extracellular water volume fraction.
The coupling is expressed through a suitable proliferation function $p(\varphi)$ multiplied by the differences
of the chemical potentials for $\varphi$ and $\psi$. The system is equipped with no-flux boundary conditions
which entails the conservation of the total mass, that is, the spatial average of $\varphi+\psi$.
Here we prove the existence of a weak solution to the associated Cauchy problem, provided that
the potential $F$ and $p$ satisfy sufficiently general conditions. Then we show that the weak solution
is unique and continuously depends on the initial data, provided that $p$ satisfies slightly stronger
growth restrictions. Also, we demonstrate the existence of a strong solution and that any weak solution
regularizes in finite time. Finally, we prove the existence of the global attractor in a phase space
characterized by an a priori bounded energy.
\end{abstract}

\section{Introduction}

Modeling tumor growth dynamic has recently become a major issue in applied mathematics (see, for instance, \cite{CL,Letal}, cf. also \cite{AM,OPH}).
The models can be divided into two broad categories: continuum models and discrete or cellular automata models (however, see, e.g., \cite[Chap.7]{CL} for hybrid continuum-discrete models). Concerning the former ones, the necessity of dealing with multiple interacting constituents has led to consider
diffuse-interface models based on continuum mixture theory (see, for instance, \cite{CLLW,OHP,WLFC} and references therein, cf. also \cite{CBCB,Fetal,HPZO}). Such models generally consist of Cahn-Hilliard equations with transport and reaction terms which govern various types of cell concentrations. The reaction terms depend on the nutrient concentration (e.g., oxygen) which obeys to an advection-reaction-diffusion equation coupled with the Cahn-Hilliard equations. The cell velocities satisfy a generalized Darcy's (or Brinkman's) law where, besides the pressure gradient, there is also the so-called Korteweg force due to the cell concentration.
Numerical simulations of diffuse-interface model for tumor growth have been carried out in several papers (see, for instance, \cite[Chap.8]{CL} and references therein). Nonetheless, a rigorous mathematical analysis of the resulting systems of differential equations is still in its infancy. In particular, to the best of our knowledge, the first related papers are concerned with the so-called Cahn-Hilliard-Hele-Shaw system (see \cite{LTZ}, cf. also \cite{BCG,WW,WZ}) in which the nutrient is neglected. Moreover, a very recent contribution (see \cite{CGH}) is devoted to analyzing an approximation of a model recently proposed in \cite{HZO} (see also \color{black}\cite{HKNV,WZZ}\color{black}). In this model, velocities are set to zero and the state variables are reduced to the tumor cell fraction $\varphi$ and the nutrient-rich extracellular water fraction $\psi$. The corresponding PDE system is given by
\begin{align}
&\varphi_t=\Delta\mu+p(\varphi)(\psi-\mu)\label{eq1}\\
&\mu=-\Delta\varphi+F'(\varphi)\label{eq2}\\
&\psi_t=\Delta\psi-p(\varphi)(\psi-\mu)\label{eq3}
\end{align}
in $\Omega\times(0,\infty)$, where $\Omega\subset\mathbb{R}^3$ is a bounded smooth domain.
Here $F$ is the typical double-well associated with the Ginzburg-Landau free-energy functional,
while $p$ is a proliferation function which must be nonnegative and may have, for instance, the
form $p(s)=p_0(1-s^2)\chi_{[-1,1]}(s)$ for $s\in\mathbb{R}$, $p_0>0$. Here $\chi_{[-1,1]}$ represents
the indicator function of $[-1,1]$. \color{black} However, in this paper we suppose $p$ to be, at least, Lipschitz continuous, but we allow it to satisfy a suitable growth condition (cf. \eqref{assp}).
\color{black} Also, it is worth observing that more general potentials $F$,
possibly depending on $\psi$ as well, might be taken into account since they are relevant from the modeling viewpoint
(cf.~\cite{ HZO} and references therein).
This could be the subject of a future work. \color{black}

System \eqref{eq1}--\eqref{eq3} is equipped with the no-flux boundary conditions
\begin{align}
&\partial_n\varphi=\partial_n\mu=\partial_n\psi=0\quad\mbox{on }\partial\Omega\times(0,\infty),\label{bcs}
\end{align}
and initial conditions
\begin{align}
&\varphi(0)=\varphi_0,\qquad\psi(0)=\psi_0\quad\mbox{in }\Omega .\label{ics}
\end{align}
In \cite{CGH} the authors consider a relaxed model in which the chemical potential $\mu$ contains a viscous term $\alpha \varphi_t$, $\alpha>0$ and equation \eqref{eq1} has an additional term $\alpha\mu_t$ which requires a further initial condition.
For this model, existence and uniqueness of a variational solution is proven under very general conditions on $F$, while $p$
is supposed to be globally bounded and Lipschitz continuous. Then, imposing substantial restrictions on $F$ (e.g., polynomial
growth of order $4$), the authors prove the existence of a sequence $\{\alpha_n\}$ and a sequence of solutions which converges to a solution to problem \eqref{eq1}--\eqref{ics} as $\alpha_n$ goes to $0$. Such a solution is more regular and unique provided that $\varphi_0$ is smooth enough.

Here we want to analyze problem \eqref{eq1}--\eqref{ics} without any regularizing term. More precisely, it is not difficult to check that
system \eqref{eq1}--\eqref{eq3} with \eqref{bcs} is characterized by the total energy balance law (see \cite[(10)]{HZO})
\begin{align}
&\frac{d}{dt}\mathcal{E}(\varphi,\psi)
+\Vert\nabla\mu\Vert^2+\Vert\nabla\psi\Vert^2+\int_\Omega p(\varphi)(\mu-\psi)^2=0,
\label{enid-intro}
\end{align}
where the energy $\mathcal{E}$ is given by
\begin{align}
&\mathcal{E}(\varphi,\psi):=\frac{1}{2}\Vert\nabla\varphi\Vert^2+\frac{1}{2}\Vert\psi\Vert^2+\int_\Omega F(\varphi).\label{toten}
\end{align}
Therefore, it seems natural to find a solution assuming that the initial data have just finite energy. This is our
first result, namely, existence of a weak solution of finite energy. The assumptions on $F$ and $p$ are more general than the ones in \cite{CGH} for the case $\alpha=0$. In particular, in the present contribution $p$ can have a polynomially controlled growth. Concerning $F$, we can take any $C^2$ and $\lambda_1$-convex potential satisfying $|F'|\leq \lambda_2F+\lambda_3$ for some nonnegative constants $\lambda_1, \, \lambda_2,\,\lambda_3$. For instance, $F(s)=\exp(s)$ or $F$ with arbitrary polynomial growth. Also, with a further restriction
on the growth of $p^\prime$ and assuming $F$ to have a polynomially controlled growth,  we can establish the continuous dependence on the initial data (and so the uniqueness of weak solutions).

The proof is obtained by suitably approximating the potential $F$ with a coercive sublinear potential $F_m$ and finding an approximating solution of such a problem through a Faedo-Galerkin scheme. The crucial point then consists in obtaining appropriate a priori estimates to pass to the limit via compactness results with respect to $m$. In particular, a bootstrap argument is used in order to derive the optimal regularity estimate for $\varphi$, which is necessary in order to prove the continuous dependence estimate as well as for the analysis of the global longtime behavior. \color{black} For similar double approximation techniques the reader is referred to, e.g., \cite{EG, FGR}. \color{black}

Then we prove a regularity result which helps us to investigate the global longtime behavior of the solutions. Concerning this issue, observe that conditions \eqref{bcs} imply the conservation of the {\itshape total} mass
\begin{align}
&\int_\Omega\big(\varphi(t)+\psi(t)\big)=\int_\Omega(\varphi_0+\psi_0),\quad\forall t\geq 0.
\end{align}
However, we are not able to obtain independent global bounds for the spatial averages of $\varphi(t)$ and $\psi(t)$. On account of this fact, we can show that \eqref{eq1}--\eqref{bcs} generates a dynamical system taking as phase space a bounded set in the finite energy space with a constraint on the total mass. We can thus prove that such a system has a global attractor.

This is just a preliminary step towards the theoretical analysis of more refined models. For instance, one may include the fluid velocity either
given as a datum or satisfying a generalized Darcy's (or Brinkman's) law. Also, one should take a logarithmic potential $F$, which is physically more relevant, and nonconstant (possibly degenerate) mobility in the Cahn-Hilliard equation. On the other hand, the free energy functional may contain a nonlocal spatial interaction in place of the usual term $\vert\nabla\varphi\vert^2$ giving rise to a convolution operator acting on $\varphi$ in place of $\Delta \varphi$ in \eqref{eq2} (see, for instance, \cite{WLFC}, cf. also \cite{GL1,GL2}). These are just some examples of challenging extensions of the simplified model expressed by \eqref{eq1}--\eqref{eq3}.

\paragraph{Plan of the paper.} In Section~\ref{sec:notation} we define the notation and we recall a useful inequality.  In Section~\ref{sec:weak} we prove that Problem~\eqref{eq1}--\eqref{ics} admits a unique weak solution (which continuously depends on the data) under proper assumptions on the nonlinearities $F$ and $p$. In Section~\ref{sec:strong} we establish a regularity result for Problem \eqref{eq1}--\eqref{ics} that holds under the same condition on $p$ which ensures uniqueness. This result turns out to be crucial in order to eventually prove the existence of the global attractor.

\section{Notation and preliminaries}
\label{sec:notation}

Let $\Omega$ be a sufficiently regular, bounded domain in $\mathbb{R}^3$, let $T>0$ and set $Q=\Omega\times (0,T)$.
Then we define $H:=L^2(\Omega)$ and $V:=H^1(\Omega)$ and denote by $\Vert\cdot\Vert$, $(\cdot,\cdot)$ the norm and the scalar product in $H$,
respectively. If $X$ is a (real) Banach space, the notation $\langle\cdot,\cdot\rangle$ will be used to denote
the duality pairing between $X$ and its dual $X'$, \color{black} while $(\cdot,\cdot)_X$ will denote the scalar product in $X$\color{black}. For every $f\in V'$, $\overline{f}$ will stand for the average
of $f$ over $\Omega$, i.e., $\overline{f}:=|\Omega|^{-1}\langle f,1\rangle$. Here $|\Omega|$ is the Lebesgue measure of $\Omega$.

\color{black}
Since it is convenient to rewrite the equations \eqref{eq1} and \eqref{eq3}
as abstract equations in the framework of the Hilbert triplet $(V,H,V')$,
we introduce the Riesz isomorphism $A:V\to V'$
associated to the standard scalar product of~$V$, that is,
\begin{equation}
  \langle A u, v\rangle := (u,v)_V
  = \int_\Omega \left(\nabla u \cdot \nabla v + uv \right)
  \quad \hbox{for $u,v\in V$}.
  \label{defA}
\end{equation}
We notice that $A u=-\Delta u+u$ if $u\in D(A):=\big\{\varphi\in H^2(\Omega):\:\:\partial_{\bf n}\varphi=0\mbox{ on }\partial\Omega\big\}$
and that the restriction of $A$ to $D(A)$ is an isomorphism from $D(A)$ onto~$H$.
We also remark that
\begin{align}\nonumber
  &\langle A u , A^{-1} v^*\rangle
  =\langle v^*, u\rangle
  \quad \hbox{for every $u\in V$ and $v^*\in V'$}
  \\
\nonumber
  &\langle u^* , A^{-1} v^* \rangle
  = (u^*,v^*)_{V'}
  \quad \hbox{for every $u^*,v^*\in V'$}
\end{align}
where $(\cdot,\cdot)_{V'}$ is the dual scalar product in $V'$
associated to the standard one in~$V$,
and recall that $\langle v^*,u\rangle=\int_\Omega v^* u$ if $v^*\in H$ and we have
$$
  \frac d{dt} \|v^*\|_{V'}^2=
  = 2 \langle\partial_tv^* , A^{-1} v^* \rangle
  \quad \hbox{for every $v^*\in H^1(0,T; V')$} .
$$
\color{black}
Moreover, by a classical spectral theorem
there exist a sequence of eigenvalues $\lambda_j$ with $0<\lambda_1\leq\lambda_2\leq\cdots$ and $\lambda_j\to\infty$,
and a family of eigenfunctions $w_j\in D(A)$ such that $Aw_j=\lambda_j w_j$. The family of $w_j$
is an orthonormal basis in $H$ and it is also orthogonal in $V$ and $D(A)$.

We shall repeatedly use the following Gagliardo-Nirenberg inequality in dimension $3$
(see, e.g., \cite{BIN,Fr,Ga,Ni} for more details)

\begin{lem}
Let $1\leq p_1,p_2\leq\infty$, $0\leq r<l$ ($r,l\in\mathbb{N}$) and assume that
$$\theta:=\frac{3/m-3/p_1-r}{3/p_2-3/p_1-l}\in[r/l,1).$$
Then
\begin{align}
&
\Vert u \Vert_{W^{r,m}(\Omega)}\leq c\Vert u \Vert_{L^{p_1}(\Omega)}^{1-\theta}\Vert u\Vert_{W^{l,p_2}(\Omega)}^\theta,\quad \forall u\in W^{l,p_2}\cap L^{p_1}(\Omega).\label{GN}
\end{align}
\end{lem}

\section{Existence and uniqueness of weak solutions}
\label{sec:weak}

In this section we prove that Problem \eqref{eq1}--\eqref{ics} admits a weak solution, provided that
$F$ and $p$ have polynomial growth with given orders $\rho$ and $q$, respectively.
The upper bounds on $\rho$ and $q$ in Theorem \ref{existence} ensure the existence of a weak solution with optimal regularity for $\varphi$,
i.e., $\varphi\in L^2(0,T;H^3(\Omega))$. Such assumptions can be relaxed if only existence of the weak solution is required
(cf. Corollary \ref{cor}). An additional restriction on the proliferation function $p$  allows
us to prove uniqueness as well as a continuous dependence estimate on the initial data for weak solutions.
In any case, our assumptions on $F$ and $p$ are more general than those made in \cite{HZO} (cf. also \cite{CGH} when $\alpha=0$).

Let us begin with the existence result, which will be proven, for the case where the growth $\rho$ of $F$ is greater
than $4$, by means of a double approximation
procedure, namely by first exploiting the Faedo-Galerkin scheme to prove existence for \color{black}$\rho\leq 4$ \color{black} and then
by approximating $F$ with a sequence of potentials having growth which is at most $4$.

 The assumptions we need for the existence are the following
\begin{description}
{\color{black}

\item[(F)] $F\in C^2(\mathbb{R})$  can be written as
           \begin{align}
           &F(s)=F_0(s)+\lambda(s)
           \end{align}
where $F_0\in C^2(\mathbb{R})$ and $\lambda\in C^2(\mathbb{R})$ satisfies
$|\lambda''(s)|\leq\alpha$, for all $s\in\mathbb{R}$, and for some constant $\alpha\geq 0$. Moreover, we assume
           \begin{align}
           &c_1(1+|s|^{\rho-2})\leq F_0''(s)\leq c_2(1+|s|^{\rho-2}),\label{F1}\\
           & F(s)\geq c_3|s|-c_4,\label{F2}
           \end{align}
for all $s\in\mathbb{R}$, with $c_1,c_2,c_3>0$, $c_4\in\mathbb{R}$ and with $\rho\in [2,6)$.
}
\item[(P)] $p\in C^{0,1}_{loc}(\mathbb{R})$ satisfies
\begin{align}
&0\leq p(s)\leq {\color{black}c_5}(1+|s|^q),\label{assp}
\end{align}
for all $s\in\mathbb{R}$, with {\color{black}$c_5>0$} and with $q\in [1,9)$.

\end{description}

Before stating the existence result, let us introduce the definition of weak solution to Problem
\eqref{eq1}--\eqref{ics}.
\begin{defn}\label{def}
Let $\varphi_0\in V$, $\psi_0\in H$ and $0<T<\infty$ be given.
Then, a pair $[\varphi,\psi]$ is a weak solution to \eqref{eq1}--\eqref{ics}
on $[0,T]$ if
\begin{align}
&\varphi\in L^\infty(0,T;V)\cap L^2(0,T;H^2(\Omega)),\quad\varphi_t\in L^r(0,T;D(A^{-1})),\label{basicreg1}\\
&\mu:=-\Delta\varphi+F'(\varphi)\in L^2(0,T;V),\label{basicreg2}\\
&\psi\in L^\infty(0,T;H)\cap L^2(0,T;V),\quad\psi_t\in L^r(0,T;D(A^{-1})),\label{basicreg3}
\end{align}
for some $r>1$, and the following identities are satisfied
\begin{align}
&\langle\varphi_t,\chi\rangle+(\nabla\mu,\nabla\chi)=\big(p(\varphi)(\psi-\mu),\chi\big),\label{wf1}\\
&\langle\psi_t,\xi\rangle+(\nabla\psi,\nabla\xi)=-\big(p(\varphi)(\psi-\mu),\xi\big),\label{wf2}
\end{align}
for all $\chi,\xi\in D(A)$ and for almost all $t\in(0,T)$, together with the initial conditions \eqref{ics}.

\end{defn}
\begin{oss}
{\upshape
Notice that the regularity properties of weak solution imply that
$$\varphi\in C_w([0,T];V),\quad\psi\in C_w([0,T];H).$$
Hence, the initial conditions \eqref{ics} make sense. \color{black} Moreover, we point out that the required regularity
for $\partial\Omega$ in order to prove our theorems is at least $C^{2,1}$. For instance, we need some regularity for the
eigenfunctions (see proof of Theorem \ref{existence}) as well as when we deduce that $\varphi\in L^2(0,T;H^2(\Omega))$ (cf.~\eqref{basicreg1}).  \color{black}
}
\end{oss}

\begin{thm}
\label{existence}
Assume that (F) and (P) are satisfied. Let $\varphi_0\in V$ and $\psi_0\in H$. Then, for every $T>0$, Problem
\eqref{eq1}--\eqref{ics} admits a weak solution on $[0,T]$ such that
\begin{align}
&\varphi\in L^2(0,T;H^3(\Omega)),\label{optreg}\\
&F(\varphi)\in L^\infty\big(0,T;L^1(\Omega)\big),\quad\sqrt{p(\varphi)}(\mu-\psi)\in L^2(0,T;H),\label{basicreg4}
\end{align}
which satisfies the following energy inequality
\begin{align}
&\mathcal{E}(\varphi,\psi)
+\int_0^t\big(\Vert\nabla\mu\Vert^2+\Vert\nabla\psi\Vert^2\big)d\tau+\int_0^t\int_\Omega p(\varphi)(\mu-\psi)^2
\leq
\mathcal{E}(\varphi_0,\psi_0),\qquad\forall t>0,
\label{enineq}
\end{align}
where $\mathcal{E}$ is given by \eqref{toten}. Furthermore, if $q\leq 4$, then we have
\begin{align}
&\varphi_t,\psi_t\in L^2(0,T;V'),\label{basicreg5}
\end{align}
and \eqref{enineq} holds with the {\color{black}equality} sign. Moreover, in this case the weak formulation \eqref{wf1}, \eqref{wf2} is satisfied also for all $\chi,\xi\in V$.

\end{thm}

{\color{black}
\begin{oss}
{\upshape
The bound $\rho< 6$ is required only to gain the optimal regularity
$\varphi\in L^2(0,T;H^3(\Omega))$. Actually, we should only require $\rho\leq 6$. However, due to technical reasons, we are not able to perform our bootstrap technique in the case $\rho=6$ (cf. Step II in the proof of Theorem~\ref{existence}).
Nevertheless, the existence of a weak solution (without this optimal regularity) can be proven under more general assumptions on $F$
(together with a slight restriction on $q$). In particular, for $F$ with polynomial growth of arbitrary order (see Corollary \ref{cor}).

}
\end{oss}
}

The following lemma will turn to be useful in the proof
of Theorem \ref{existence} (cf. Step II). Indeed, it allows to suitably approximate a regular
potential having general $\rho-$growth (in particular in case $\rho>4$) and satisfying conditions
\eqref{F1}, \eqref{F2} with a sequence of regular potentials
having {\color{black}quadratic} growth.

\color{black}

{\color{black}
\begin{lem}
\label{Fapproxim}

Assume that $F$ satisfies (F) with $\rho>2$.
 Then, there exists a sequence of $F_m\in C^2(\mathbb{R})$ satisfying
$$|F_m(s)|\leq \alpha_m(1+s^2),\qquad\forall s\in \mathbb{R},$$
for some constant $\alpha_m\geq 0$,
such that $F_m(s)\to F(s)$ pointwise for all $s\in\mathbb{R}$ as $m\to+\infty$
and fulfilling, for every $m\in\mathbb{N}$,  the bounds
\begin{align}
&|F_m(s)|\leq k_0|F(s)|,\quad |F_m'(s)|\leq k_1|F'(s)|,\quad |F_m''(s)|\leq k_2 |F''(s)|,\quad\forall s\in\mathbb{R},
\label{Fm}
\end{align}
and the equi-coercivity conditions
\begin{align}
&F_m(s)\geq k_3 s^2-k_4,\qquad F_m''(s)\geq -k_5\quad\forall s\in\mathbb{R},\label{equicoer}
\end{align}
where $k_i$, $i=0,\cdots, 5$ are some positive constants depending on $F$ and $\rho$ only.
\end{lem}
\begin{proof}
Without loss of generality, we set
$F_0(0)=F_0'(0)=0$ (this condition can always be assumed
by redefining the function $\lambda$).
Set $H_0=F_0'$ and let $H_{0m}$ be the Yosida regularization of $H_0$ defined by (cf., e.g., \cite[p.~28]{brezis})
\begin{align}
&H_{0m}(s)=m\big(s-J_m(s)\big),\qquad J_m(s)=\Big(I+\frac{1}{m}H_0\Big)^{-1}(s),\qquad\forall s\in\mathbb{R}.\nonumber
\end{align}
Introduce now
\begin{align}
&F_{0m}(s)=\int_0^s H_{0m}(\sigma)d\sigma,\nonumber
\end{align}
for all $s\in\mathbb{R}$, and set
$$F_m(s)=F_{0m}(s)+\lambda(s).$$
Let us check that the sequence of $F_m$ satisfies all the stated conditions.
We shall use standard results from the theory of maximal monotone operators (applied to the single-valued
monotone function
$H_0$ defined on the whole of $\mathbb{R}$).

First, notice that $H_{0m}$ is Lipschitz continuous with Lipschitz constant equal to $m$,
and then $|H_{0m}(s)|\leq m|s|$, for all $s\in\mathbb{R}$, where we have
used the fact that $H_{0m}(0)=J_m(0)=0$, since $F_0'(0)=0$. Therefore,
$$|F_{0m}(s)|\leq \frac{1}{2}m s^2,\qquad\forall s\in\mathbb{R},$$
which implies that $F_m$ has at most quadratic growth for each $m$.

Moreover, we have $|H_{0m}(s)|\leq |H_0(s)|=|F_0'(s)|$ and also
$H_{0m}(s)\to H_0(s)=F_0'(s)$, for all $s\in\mathbb{R}$ as $m\to+\infty$.
Therefore, by the Lebesgue theorem we deduce
\begin{align}
&F_{0m}(s)\to\int_0^s F_0'(\sigma)d\sigma=F_0(s),\quad \hbox{as }m\to+\infty\nonumber
\end{align}
and this implies that $F_m(s)\to F(s)$ for all $s\in\mathbb{R}$ as $m\to+\infty$.

Next, the bound $\eqref{Fm}_1$ is immediate, since, for all $s\in\mathbb{R}$ we have
\begin{align*}
&|F_{0m}(s)|\leq\Big|\int_0^s |H_{0m}(\sigma)|d\sigma\Big|
\leq\Big|\int_0^s |F_0'(\sigma)|d\sigma\Big|=F_0(s),\\
&|F_{0m}'(s)|=|H_{0m}(s)|\leq |H_0(s)|=|F_0'(s)|.
\end{align*}
Also, we can take the growth condition \eqref{F1} into account. As far as $\eqref{Fm}_3$ is
concerned, notice first that we have $F_{0m}''(s)=H_{0m}'(s)=m(1-J_m'(s))$.
Moreover, from
\begin{align}
&r+\frac{1}{m} H_0(r)=s\Longleftrightarrow r=J_m(s)\label{resolvent}
\end{align}
we deduce
$$J_m'(s)=\frac{1}{1+\frac{1}{m}F_0''\big(J_m(s)\big)}.$$
Hence we have
\begin{align}
&F_{0m}''(s)=\frac{F_0''\big(J_m(s)\big)}{1+\frac{1}{m}F_0''\big(J_m(s)\big)}
\leq F_0''\big(J_m(s)\big)
\leq c_2\big(1+|J_m(s)|^{\rho-2}\big)\nonumber\\
&\leq c_2\big(1+|s|^{\rho-2}\big)
\leq\frac{c_2}{c_1} F_0''(s),\qquad\forall s\in\mathbb{R}.\nonumber
\end{align}
Bound $\eqref{Fm}_3$ then follows from this last estimate and \eqref{F1}.

Furthermore, we have
\begin{align}
&F_{0m}''(s)\geq\frac{F_0''\big(J_m(s)\big)}{1+F_0''\big(J_m(s)\big)}\geq\frac{c_1}{1+c_1},\qquad\forall s\in\mathbb{R},\quad\forall m,
\nonumber
\end{align}
and this, together with the assumption on $\lambda$, yields $\eqref{equicoer}_2$.
Let us finally check that also $\eqref{equicoer}_1$ holds.
To this purpose we first recall the following property:
let $\beta>0$ and $\gamma\in\mathbb{R}$ be two constants such that
$$F_0(s)\geq\beta s^2-\gamma,\qquad\forall s\in\mathbb{R}.$$
Then, we have
$$F_{0m}(s)\geq\frac{\beta}{2}s^2-\gamma,\qquad\forall s\in\mathbb{R},$$
and for all $m\geq m_0(\beta)$. We report the proof for the reader's convenience.
Indeed, observe that
\begin{align}
&F_{0m}(s)=\frac{1}{2m}H_{0m}^2(s)+F_0(J_m(s))\geq\frac{1}{2m}H_{0m}^2(s)+\beta J_m^2(s)-\gamma\nonumber\\
&=\frac{1}{2m}H_{0m}^2(s)+\beta\Big(s-\frac{1}{m} H_{0m}(s)\Big)^2-\gamma
\geq\frac{1}{4m}H_{0m}^2(s)+\beta\Big(1-\frac{4\beta}{m}\Big) s^2-\gamma\nonumber\\
&\geq\frac{\beta}{2} s^2-\gamma,\nonumber
\end{align}
provided we choose $m\geq m_0(\beta):=8\beta$.

Using now \eqref{F1} and the fact that $\rho>2$, we can write
\begin{align}
&F_0(s)\geq\hat{c}_1(|s|^\rho+s^2)\geq \frac{\hat{c}_1}{\delta}s^2-C_\delta,\nonumber
\end{align}
where $\delta>0$ will be fixed later. By employing the property recalled above and the fact that
we have $\lambda(s)\geq -\alpha s^2-\widetilde{\alpha}$, we
deduce
\begin{align}
&F_m(s)=F_{0m}(s)+\lambda(s)\geq \Big(\frac{\hat{c}_1}{2\delta}-\alpha\Big) s^2-C_\delta,\nonumber
\end{align}
which holds for all $s\in\mathbb{R}$ and for all $m\geq 8\hat{c}_1/\delta$.
Let us choose, e.g., $\delta=\hat{c}_1/2(1+\alpha)$. Therefore we have
\begin{align}
&F_m(s)\geq s^2-C,\qquad\forall s\in\mathbb{R},\qquad \forall m\geq m_0,\nonumber
\end{align}
where $m_0=16(1+\alpha)$. Hence, $\eqref{equicoer}_1$ is proven and the proof of the lemma
is complete.\end{proof}
}

\color{black}
\smallskip

\noindent
{\it Proof of Theorem~\ref{existence}.}
{\itshape Step I (case {\color{black}$\rho\leq4$}).}

Let us first prove the existence of a weak solution
with optimal regularity \eqref{optreg} under the assumption that $F$ has growth $4$ at most.
We shall use a Faedo-Galerkin approximation method. Let us then take the family $\{w_j\}_{j\geq 1}$
of the eigenfunctions of $A$
as a Galerkin basis in $V$,
and let $P_n$ be the orthogonal projectors in $H$ onto the $n$-dimensional subspace $\mathcal{W}_n:=\langle w_1,\cdots w_n\rangle$
spanned by the first $n$ eigenfunctions. For $n\in\mathbb{N}$ fixed, we look for three functions of the form
\begin{align*}
\varphi_n(t):=\sum_{k=1}^n a^n_k(t)w_k,\quad\psi_n(t):=\sum_{k=1}^n b^n_k(t)w_k,\quad\mu_n(t):=\sum_{k=1}^n c^n_k(t)w_k
\end{align*}
that solve the following approximating problem
\begin{align}
&(\varphi_n',w_j)+(\nabla\mu_n,\nabla w_j)=\big(p(\varphi_n)(\psi_n-\mu_n),w_j\big),\label{pbapp1}\\
&(\mu_n,w_j)=(\nabla\varphi_n,\nabla w_j)+\big(F'(\varphi_n),w_j\big),\label{pbapp2}\\
&(\psi_n',w_j)+(\nabla\psi_n,\nabla w_j)=-\big(p(\varphi_n)(\psi_n-\mu_n),w_j\big),\label{pbapp3}\\
&\varphi_n(0)=\varphi_{0n},\quad\psi_n(0)=\psi_{0n},\label{pbapp4}
\end{align}
for $j=1,\cdots,n$, where $\varphi_{0n}:=P_n\varphi_0$ and $\psi_{0n}:=P_n\psi_0$
(prime denote the derivative with respect to time).


It is easy to see that solving the approximate problem
\eqref{pbapp1}--\eqref{pbapp4} is equivalent to solving a Cauchy problem for a system
of $2n$ ordinary differential equations in the $2n$ unknowns $a^n_j$, $b^n_j$.
Since $F'\in C^1$ and $p\in C^{0,1}_{loc}$, the Cauchy-Lipschitz theorem
ensures that there exists $T^\ast_n\in(0,\infty]$ such that this system has a unique maximal solution ${\bf a}^n:=(a^n_1,\cdots,a^n_n)$,
${\bf b}^n:=(b^n_1,\cdots,b^n_n)$ on $[0,T^\ast_n)$ with ${\bf a}^n,{\bf b}^n\in C^1([0,T^\ast_n);\mathbb{R}^n)$.
Hence, the approximate problem \eqref{pbapp1}--\eqref{pbapp4} admits a unique solution
$\varphi_n,\psi_n,\mu_n\in C^1([0,T^\ast_n);\mathcal{W}_n)$.

We now deduce the basic estimates on the sequence of approximating solutions. In particular, these estimates will
guarantee that $T^\ast_n=\infty$ for every $n\in\mathbb{N}$.

Multiply then \eqref{pbapp1} by $c^n_j$, \eqref{pbapp2} by $a^n_j{'}$, \eqref{pbapp3} by $b^n_j$
and sum the resulting identities over $j=1,\cdots,n$. We get the following energy identity
satisfied by the solution of the approximate problem
\begin{align}
&\frac{d}{dt}\Big(\frac{1}{2}\Vert\nabla\varphi_n\Vert^2+\frac{1}{2}\Vert\psi_n\Vert^2+\int_\Omega F(\varphi_n)\Big)
+\Vert\nabla\mu_n\Vert^2+\Vert\nabla\psi_n\Vert^2+\int_\Omega p(\varphi_n)(\mu_n-\psi_n)^2=0.
\label{enidapp}
\end{align}
By integrating \eqref{enidapp} in time between $0$ and $t$, using (F), (P) and the assumptions on the initial data
we immediately deduce the following estimates
\begin{align}
&\Vert\varphi_n\Vert_{L^\infty(0,T;V)}\leq C, \quad \Vert\psi_n\Vert_{L^\infty(0,T;H)\cap L^2(0,T;V)}\leq C,
\label{basicest1}\\
&\Vert\nabla\mu_n\Vert_{L^2(0,T;H)}\leq C,\quad \Vert\sqrt{p(\varphi_n)}(\psi_n-\mu_n)\Vert_{L^2(0,T;H)}\leq C,
\label{basicest2}\\
&\Vert F(\varphi_n)\Vert_{L^\infty(0,T;L^1(\Omega))}\leq C.\label{basicest3}
\end{align}
where henceforth $C=C\big(\Vert\varphi_0\Vert_V,\Vert\psi_0\Vert\big)$ denotes a nonnegative constant
depending on the norms of the initial data (and on $F$, $\Omega$).

Let us now control the sequence of the averages of $\mu_n$. From \eqref{pbapp2} we get
\begin{align}
&|(\mu_n,1)|=|(F'(\varphi_n),1)|\leq {\color{black}c_6}\big(F(\varphi_n),1\big)+{\color{black}c_7}\leq C,\label{avcontr}
\end{align}
where {\color{black}$c_6,c_7$} are two nonnegative constants depending only on $F$, $\Omega$ and where we have used assumption
(F)
and \eqref{basicest3}.
Therefore, the sequence of $\overline{\mu}_n$ is bounded in $L^\infty(0,T)$ and this bound, together
with the first of \eqref{basicest2} yields
\begin{align}
&\Vert\mu_n\Vert_{L^2(0,T;V)}\leq C.\label{basicregapp2}
\end{align}


We now prove that the sequence of $\varphi_n$ is controlled in $L^\infty(0,T;V)\cap L^2(0,T;H^3(\Omega))$.
Indeed, notice first that \eqref{pbapp2} can be written as
\begin{align}
&\mu_n=-\Delta\varphi_n+P_n F'(\varphi_n).\label{pbapp2bis}
\end{align}
Observe now that $\Vert P_n F'(\varphi_n)\Vert\leq \Vert F'(\varphi_n)\Vert$. Thus, the sequence
of $\varphi_n$ is bounded in $L^\infty(0,T;L^6(\Omega))$,  we deduce from \eqref{F1} the bound
\begin{align}
&\Vert F'(\varphi_n)\Vert_{L^\infty(0,T;H)}\leq C.\label{est37}
\end{align}
Hence, \eqref{pbapp2bis} and \eqref{basicregapp2}
entail that the sequence of $-\Delta\varphi_n+\varphi_n$ is bounded in $L^2(0,T;H)$ and, on account of the homogeneous
Neumann boundary condition for $\varphi_n$, a classical elliptic regularity result implies
\begin{align}
&\Vert\varphi_n\Vert_{L^\infty(0,T;V)\cap L^2(0,T;H^2(\Omega))}\leq C.\label{est16}
\end{align}
By using inequality \eqref{GN}, we deduce from \eqref{est16} that the sequence of $\varphi_n$
is bounded in $L^{10}(Q)$
and moreover the sequence of $\nabla\varphi_n$ is bounded in $L^\infty(0,T;H)\cap L^2(0,T;V)
\hookrightarrow L^{10/3}(Q)$.
On the other hand,
{\color{black}
note that $\Vert A^{1/2}u\Vert^2=(Au,u)=\Vert\nabla u\Vert^2+\Vert u\Vert^2$, for all
$u\in D(A)$. Hence $\Vert A^{1/2}u\Vert\geq \Vert\nabla u\Vert$
(which holds, by density, also for all $u\in V=D(A^{1/2})$). Therefore}
we have
$$\Vert\nabla\big(P_n F'(\varphi_n)\big)\Vert\leq\Vert A^{1/2} P_n F'(\varphi_n)\Vert=\Vert  P_n A^{1/2} F'(\varphi_n)\Vert\leq\Vert\nabla F'(\varphi_n)\Vert+\Vert F'(\varphi_n)\Vert,$$
and hence
\eqref{F1} together with \eqref{est37} and $\eqref{basicest1}_1$ entail
\begin{align*}
\Vert P_n F'(\varphi_n)\Vert_{L^2(0,T;V)}
&\leq
\Vert F''(\varphi_n)\nabla\varphi_n\Vert_{L^{2}(Q)}
+\Vert F'(\varphi_n)\Vert_{L^{2}(Q)}\nonumber\\
&\leq \Vert F''(\varphi_n)\Vert_{L^5(Q)}\Vert\nabla\varphi_n\Vert_{L^{10/3}(Q)}
+\Vert F'(\varphi_n)\Vert_{L^{2}(Q)}\nonumber\\
&\leq c\big(1+\Vert\varphi_n\Vert_{L^{10}(Q)}^2\big)\Vert\nabla\varphi_n\Vert_{L^{10/3}(Q)}+
\Vert F'(\varphi_n)\Vert_{L^{2}(Q)}\nonumber\\
&\leq C.
\end{align*}
By comparison in \eqref{pbapp2bis}, using \eqref{basicregapp2} and the elliptic regularity result again, we infer
\begin{align}
\Vert\varphi_n\Vert_{L^\infty(0,T;V)\cap L^2(0,T;H^3(\Omega))}\leq C.\label{est17}
\end{align}

We now deduce the estimates for the sequences of time derivatives
$\varphi_n'$ and $\psi_n'$. Take $\chi\in D(A)\hookrightarrow L^\infty(\Omega)$ and write it as
$\chi=\chi_1+\chi_2$, where $\chi_1=P_n\chi\in\mathcal{W}_n$ and $\chi_2\in (I-P_n)\chi\in\mathcal{W}_n^{\perp}$
(recall that $\chi_1$, $\chi_2$ are orthogonal in $H$, $V$ and $D(A)$).
Then, from \eqref{pbapp1} we have
\begin{align}
&\langle\varphi_n',\chi\rangle=\langle\varphi_n',\chi_1\rangle=-(\nabla\mu_n,\nabla\chi_1)
+\big(p(\varphi_n)(\psi_n-\mu_n),\chi_1\big),
\label{pbapp1bis}
\end{align}
and a similar identity follows from \eqref{pbapp3}.
Observe that
\begin{align*}
|\big(p(\varphi_n)(\psi_n-\mu_n),\chi_1\big)|
&\leq \Vert p(\varphi_n)\Vert_{L^{6/5}(\Omega)}\Vert\psi_n-\mu_n\Vert_{L^6(\Omega)}\Vert\chi_1\Vert_{L^\infty(\Omega)}\nonumber\\
&\leq c\Vert p(\varphi_n)\Vert_{L^{6/5}(\Omega)}\Vert\psi_n-\mu_n\Vert_{L^6(\Omega)}\Vert\chi\Vert_{D(A)}.
\end{align*}
The term $(\psi_n-\mu_n)$ is controlled in $L^2(0,T;L^6(\Omega))$, then we need to control the sequence of $p(\varphi_n)$
in $L^\sigma(0,T;L^{6/5}(\Omega))$ with some $\sigma>2$ in order to get the control of the sequences of $\varphi_n',\psi_n'$
in $L^r(0,T;D(A^{-1}))$ with some $r>1$.
{\color{black} To this aim notice that
from assumption (P) it follows
\begin{align}
&\Vert p(\varphi_n)\Vert_{L^\sigma(0,T;L^{6/5+\epsilon}(\Omega))}
\leq c(1+\Vert\varphi_n\Vert_{L^{\sigma q}(0,T;L^{6q/5+\epsilon q}(\Omega))}^q),
\label{est20}
\end{align}
where $\sigma> 2$ and $\epsilon> 0$.
On the other hand,}
 we know that the sequence of $\varphi_n$ is bounded in $L^\infty(0,T;V)\cap L^2(0,T;H^3(\Omega))$
(cf. \eqref{est17}), and, thanks to inequality \eqref{GN}, we have the following embedding
\begin{align}
&L^\infty(0,T;V)\cap L^2(0,T;H^3(\Omega))\hookrightarrow
L^{8\theta/(\theta-6)}(0,T;L^\theta(\Omega)),\quad \mbox{for } 6\leq\theta\leq\infty.\label{estro}
\end{align}
Hence, choosing $\theta=54/5$, we obtain
\begin{align}
&\Vert\varphi_n\Vert_{L^{18}(0,T;L^{54/5}(\Omega))}\leq C.\label{est39}
\end{align}
{\color{black}
Recalling that $q<9$, we can then fix $\sigma>2$ and $\epsilon>0$ such that
$\sigma q\leq 18$ and $6q/5+\epsilon q\leq 54/5$ (both $\sigma$ and $\epsilon$ depending on $q$). Thus we have
$L^{18}(0,T;L^{54/5}(\Omega))\hookrightarrow L^{\sigma q}(0,T;L^{6q/5+\epsilon q}(\Omega))$.
Therefore, on account of \eqref{est20} and \eqref{est39}, we get the desired control of
$p(\varphi_n)$ in $L^\sigma(0,T;L^{6/5}(\Omega))$ with some $\sigma>2$.
}
Summing up, we have proven the following bounds
\begin{align}
&\Vert\varphi_n'\Vert_{L^r(0,T;D(A^{-1}))}\leq C,\quad\Vert\psi_n'\Vert_{L^r(0,T;D(A^{-1}))}\leq C,
\quad\mbox{for some }r>1,\label{est26}
\end{align}
where we have used \eqref{est20} and $\eqref{basicest2}_1$ in \eqref{pbapp1bis}
to get the first bound and \eqref{est20} and $\eqref{basicest1}_2$ to obtain the second bound.


We now deduce from estimates \eqref{basicest1}, \eqref{basicregapp2}, \eqref{est17} and \eqref{est26}
the existence of three functions $\varphi\in L^\infty(0,T;V)\cap L^2(0,T;H^3(\Omega))$,
$\psi\in L^\infty(0,T;H)\cap L^2(0,T;V)$ and $\mu\in L^2(0,T;V)$, with $\varphi_t,\psi_t\in L^r(0,T;D(A^{-1}))$
which are the (weak) limits (up to subsequences) of $\varphi_n$, $\psi_n$, $\mu_n$ and $\varphi_n',\psi_n'$, respectively.
In order to pass to the limit in the approximate problem, we first observe that thanks to the compact
embedding
\begin{align*}
&L^\infty(0,T;V)\cap W^{1,r}(0,T;D(A^{-1}))\hookrightarrow\hookrightarrow C([0,T];L^\kappa(\Omega)),
\quad 2\leq\kappa<6
\end{align*}
given by the Aubin-Lions lemma (see, e.g., \cite{L}), we deduce that, up to a subsequence,
 $\varphi_n\to\varphi$ pointwise almost everywhere in $Q=\Omega\times(0,T)$.
Then, since $(\psi_n-\mu_n)$ converges weakly to $(\psi-\mu)$ in $L^2(0,T;L^6(\Omega))$,
in order to pass to the limit in $\big(p(\varphi_n)(\psi_n-\mu_n),w_j\big)$
on the right hand side of \eqref{pbapp1} and \eqref{pbapp3} it is enough that $p(\varphi_n)$
converges strongly to $p(\varphi)$ in $L^2(0,T;L^{6/5}(\Omega))$
{\color{black} (up to a subsequence). But we know that $p(\varphi_n)\to p(\varphi)$
pointwise almost everywhere in $Q$ and furthermore, from \eqref{est20}, \eqref{est39}
and the from embedding $L^{18}(0,T;L^{54/5}(\Omega))\hookrightarrow L^{\sigma q}(0,T;L^{6q/5+\epsilon q}(\Omega))$
(with $\sigma>2$ and $\epsilon>0$ fixed as above),
we have $p(\varphi_n)\rightharpoonup p(\varphi)$ weakly in $L^\sigma(0,T;L^{6/5+\epsilon}(\Omega))$. Hence we deduce}
\begin{align}
&p(\varphi_n)\to p(\varphi),\quad\mbox{strongly in } L^2(0,T;L^{6/5}(\Omega)).
\label{strconvp}
\end{align}

This convergence, combined with the weak convergence $(\mu_n-\psi_n)\rightharpoonup (\mu-\psi)$ in $L^2(0,T;L^6(\Omega))$,
allows us to pass to the limit in the nonlinear term on the right hand side of \eqref{pbapp1} and \eqref{pbapp3}
{\color{black} (recall that $w_j\in C^1(\overline{\Omega})$, assuming that $\partial\Omega$ is smooth enough, e.g., $C^{2,1}$)}.
By means of the convergences deduced above
we can therefore pass to the limit in the approximate problem \eqref{pbapp1}--\eqref{pbapp4}
and deduce that $\varphi,\psi,\mu$ satisfy \eqref{wf1}--\eqref{wf2}.
The argument is standard and the details are left to the reader.

The energy inequality \eqref{enineq} can be proven by integrating in time
\eqref{enidapp} between $0$ and $t$ and passing to the limit as $n\to\infty$
in the resulting identity. The only nontrivial point is the following inequality
\begin{align}
&\int_0^t\int_\Omega p(\varphi)(\mu-\psi)^2\leq\liminf_{n\to\infty}\int_0^t\int_\Omega p(\varphi_n)(\mu_n-\psi_n)^2.
\label{liminf}
\end{align}
We know from \eqref{estro} written for $\theta=14$,
that the sequence of $\varphi_n$ is bounded in $L^{14}(Q)$ and hence, on account of (P), the sequence of
$\sqrt{p(\varphi_n)}$ is bounded in $L^{28/q}(Q)$. Since $\varphi_n\to\varphi$ also pointwise almost everywhere in $Q$,
then we have $\sqrt{p(\varphi_n)}\to\sqrt{p(\varphi)}$ strongly in $L^\gamma(Q)$, for every $\gamma<28/q$.
In particular we have $\sqrt{p(\varphi_n)}\to\sqrt{p(\varphi)}$ strongly in $L^3(Q)$.
Therefore, we have
\begin{align*}
\sqrt{p(\varphi_n)}(\mu_n-\psi_n)\rightharpoonup\sqrt{p(\varphi)}(\mu-\psi),\quad\mbox{ in }L^{6/5}(Q),
\end{align*}
and, due $\eqref{basicest2}_2$, this last weak convergence is also in $L^2(Q)$.
Hence, \eqref{liminf} follows.

Moreover, if $q\leq 4$ we can easily deduce the regularity $\varphi_t,\psi_t\in L^2(0,T;V')$
by comparison in the variational formulation of \eqref{eq1} and \eqref{eq3}.
Indeed, estimating the term $p(\varphi)(\psi-\mu)$ in $V'$, we get
\begin{align}
&\Vert p(\varphi)(\psi-\mu)\Vert_{V'}\leq c\Vert p(\varphi)\Vert_{L^{3/2}(\Omega)}\Vert\psi-\mu\Vert_{L^6(\Omega)}.
\label{est21}
\end{align}
But, since $q\leq 4$ and $\varphi\in L^\infty(0,T;L^6(\Omega))$, then assumption (P) implies that we have $p(\varphi)\in L^\infty(0,T;L^{3/2}(\Omega))$
and therefore, on account of \eqref{basicreg2} and of $\eqref{basicreg3}_1$, \eqref{est21} entails
\begin{align*}
&p(\varphi)(\psi-\mu)\in L^2(0,T;V').
\end{align*}
Hence, \eqref{basicreg5} follows immediately.

{\color{black}
Finally,
let us take $\chi=\mu$ and $\xi=\psi$ in the variational formulation
\eqref{wf1}, \eqref{wf2}
of \eqref{eq1} and \eqref{eq3} (with test functions $\chi,\xi$ now in $V$), respectively,
and sum the resulting identities.
The choices for $\chi$ and $\xi$ are allowed since we have $\mu,\psi\in L^2(0,T;V)$.
Next, let us recall \eqref{basicreg5} for $\varphi_t,\psi_t$,
\eqref{optreg} and \eqref{basicreg3} for $\varphi,\psi$,
and the chain rule applied to the product $\langle\varphi_t,F'(\varphi)\rangle$,
noting that $F'(\varphi)\in L^2(0,T;V)$, to write the identities
\begin{align}
&\langle\varphi_t,\mu\rangle=\frac{1}{2}\frac{d}{dt}\Vert\nabla\varphi\Vert^2+\frac{d}{dt}\int_\Omega F(\varphi),\qquad
\langle\psi_t,\psi\rangle=\frac{1}{2}\frac{d}{dt}\Vert\psi\Vert^2.
\label{auxid}
\end{align}
Here we have used \cite[Proposition 4.2]{CKRS} and the fact that \eqref{F1} ensures that $F$ is a quadratic
perturbation of a convex function. Observe that the first term on the right hand side of $\eqref{auxid}_1$ can be justified by means of
a regularization argument which employs the time convolution of $\varphi$ by a family of mollifiers
(see, e.g., proof of \cite[Lemma 4.1]{TInf}).
Summing up, we obtain
\begin{align}
&\frac{d}{dt}\Big(\frac{1}{2}\Vert\nabla\varphi\Vert^2+\frac{1}{2}\Vert\psi\Vert^2+\int_\Omega F(\varphi)\Big)
+\Vert\nabla\mu\Vert^2+\Vert\nabla\psi\Vert^2+\int_\Omega p(\varphi)(\mu-\psi)^2=0.
\label{enid}
\end{align}
}
By integrating the energy identity \eqref{enid} in time between $0$ and $t$  we
deduce \eqref{enineq} with the equal sign for all $t>0$.
This completes the proof of the theorem for the case {\color{black}$\rho\leq4$}.


{\itshape Step II {\color{black}(case $4<\rho<6$).}}

\color{black}
In this case we first approximate the potential $F$ with a sequence of potentials $F_m\in C^2(\mathbb{R})$
satisfying the conditions stated in Lemma \ref{Fapproxim}.

\color{black}

Let us now consider problem \eqref{eq1}--\eqref{ics} with $F$ replaced by $F_m$ and call it Problem P$_m$.
Since $F_m$ satisfies condition (F) with $\rho\leq 4$
\color{black} {\color{black}(each $F_m$ has quadratic growth on
$\mathbb{R}$)}
\color{black}
then, for each $m\in\mathbb{N}$,
Step I ensures the existence of a weak solution $[\varphi_m,\psi_m]$ to Problem P$_m$
such that $\varphi_m\in L^\infty(0,T;V)\cap L^2(0,T;H^3(\Omega))$,
$\psi_m\in L^\infty(0,T;H)\cap L^2(0,T;V)$,
$\mu_m\in L^2(0,T;V)$ and satisfying the energy inequality \eqref{enineq}.

Due to \eqref{enineq} (written for each solution $\varphi_m, \psi_m$
\color{black} with $F_m$ in place of $F$\color{black}), assumptions (F) and (P), $\eqref{Fm}_1$ \color{black} and \eqref{equicoer}, \color{black} we can argue as for the Faedo-Galerkin approximating solutions $[\varphi_n,\psi_n]$
(cf. Step I) and we can still recover the basic estimates \eqref{basicest1}, \eqref{basicregapp2} for the sequences of $\varphi_m$ and $\psi_m$ (notice that in Problem P$_m$ the initial conditions are not approximated).

We now show that the sequence $\varphi_m$ is still controlled in $L^\infty(0,T;V)\cap L^2(0,T;H^3(\Omega))$.
This bound will be achieved through an iteration argument.

{\color{black}
Notice first that the sequence $\varphi_m$ is bounded in $L^\infty(0,T;V)\cap L^2(0,T;H^2(\Omega))$.
Indeed, by multiplying the identity $\mu_m=-\Delta\varphi_m+F_m'(\varphi_m)$
by $\Delta\varphi_m$ we obtain
\begin{align}
&\Vert\Delta\varphi_m\Vert^2=-(\mu_m,\Delta\varphi_m)+(F_m'(\varphi_m),\Delta\varphi_m)\nonumber\\
&\leq\frac{1}{2}\Vert\mu_m\Vert^2+\frac{1}{2}\Vert\Delta\varphi_m\Vert^2-\int_\Omega F_m''(\varphi_m)|\nabla\varphi_m|^2. \nonumber
\end{align}
By using $\eqref{equicoer}_2$, this last estimate yields
\begin{align}
&\Vert\Delta\varphi_m\Vert^2\leq\Vert\mu_m\Vert^2+2k_5\Vert\nabla\varphi_m\Vert^2.\label{est46}
\end{align}
The desired bound of $\varphi_m$ in $L^\infty(0,T;V)\cap L^2(0,T;H^2(\Omega))$ then follows from \eqref{est46} by
applying the basic estimates \eqref{basicest1}, \eqref{basicregapp2} and elliptic regularity.

Using the obtained bound and interpolation (cf. \eqref{est44} below), we see that
the sequence of $\varphi_m$ is bounded in $L^{2(\rho-1)}(0,T;L^{6(\rho-1)/(\rho-3)}(\Omega))$ as well.
Hence, \eqref{F1} together with $\eqref{Fm}_2$ imply that
the sequence of $F_m'(\varphi_m)$ is bounded in $L^2(0,T;L^{6/(\rho-3)}(\Omega))$.
Therefore, from \eqref{eq2} and \eqref{basicregapp2}
we infer that the sequence of $-\Delta\varphi_m+\varphi_m$ is bounded in $L^2(0,T;L^{6/(\rho-3)}(\Omega))$.
Then, by using elliptic regularity theory (see, e.g., \cite{Ag, GT, Ne}) we get
\begin{align}
&\Vert\varphi_m\Vert_{L^\infty(0,T;V)\cap L^2(0,T;W^{2,\frac{6}{\rho-3}}(\Omega))}\leq C.\label{est27}
\end{align}

}
Thanks to inequality \eqref{GN},  we deduce from \eqref{est27} that the sequence of $\varphi_m$ is bounded in {\color{black} $L^{2(11-\rho)}(Q)$}.
Moreover, $\nabla\varphi_m$ is bounded in {\color{black}$L^\infty(0,T;H)\cap L^2(0,T;W^{1,6/(\rho-3)}(\Omega))
\hookrightarrow L^{2(11-\rho)/3}(Q)$}. Therefore, using $\eqref{Fm}_3$ and \eqref{F1} we get
\begin{align*}
\Vert\nabla F_m'(\varphi_m)\Vert_{L^{s_0}(Q)}&
\leq k_2\Vert F''(\varphi_m)\nabla\varphi_m\Vert_{L^{s_0}(Q)}\nonumber\\
&\leq k_2 \Vert F''(\varphi_m)\Vert_{{\color{black}L^{2(11-\rho)/(\rho-2)}(Q)}}
\Vert\nabla\varphi_m\Vert_{{\color{black}L^{2(11-\rho)/3}(Q)}}\nonumber\\
&\leq c\big(1+\Vert\varphi_m\Vert_{{\color{black}L^{2(11-\rho)}(Q)}}^{\rho-2}\big)
\Vert\nabla\varphi_m\Vert_{{\color{black}L^{2(11-\rho)/3}(Q)}}\nonumber\\
&\leq C,\quad {\color{black}s_0=\frac{2(11-\rho)}{\rho+1}}.
\end{align*}

{\color{black}
In addition,
we know that the sequence of $F_m'(\varphi_m)$ is bounded in $L^2(0,T;L^{6/(\rho-3)}(\Omega))$.
Let us now first consider the case $4<\rho\leq 5$. In this case we have $s_0\in [2,14/5)$
and since $F_m'(\varphi_m)$ is bounded in $L^2(0,T;L^2(\Omega))$, we obtain
\begin{align*}
&\Vert F_m'(\varphi_m)\Vert_{L^2(0,T;V)}\leq C.
\end{align*}}
By comparison in \eqref{eq2} and using \eqref{basicregapp2} and elliptic regularity again, we deduce
{\color{black}the desired bound}
\begin{align}
\Vert\varphi_m\Vert_{L^\infty(0,T;V)\cap {\color{black}L^2(0,T;H^3(\Omega))}}\leq C.\label{est17bis}
\end{align}
{\color{black}
On the other hand, if $5<\rho<6$, then $s_0\in (10/7,2)$. In this case the sequence of $F_m'(\varphi_m)$
is still bounded in $L^2(0,T;L^2(\Omega))$, but we have
\begin{align*}
&\Vert F_m'(\varphi_m)\Vert_{L^{s_0}(0,T;W^{1,s_0}(\Omega))}\leq C.
\end{align*}
By comparison in \eqref{eq2} and using \eqref{basicregapp2} and elliptic regularity again, we now deduce
\begin{align}
\Vert\varphi_m\Vert_{L^\infty(0,T;V)\cap L^{s_0}(0,T;W^{3,s_0}(\Omega))}\leq C.\label{est17tris}
\end{align}}
{\color{black}In this case} we can repeat the argument above and improve the estimates for the sequence of $\varphi_m$
by means of a bootstrap procedure performed for a finite number of steps.
Indeed, observe first that, thanks to \eqref{GN}, we have (for any $s\in (1,2]$)
\begin{align}
&\mathbb{X}_s:=L^\infty(0,T;V)\cap L^s(0,T;W^{3,s}(\Omega))\hookrightarrow L^{7s}(Q),\label{est18}\\
&\mathbb{Y}_s:=L^\infty(0,T;H)\cap L^s(0,T;W^{2,s}(\Omega))\hookrightarrow L^{\frac{7}{3}s}(Q)\label{est19}.
\end{align}
Taking {\color{black}\eqref{est17tris}}--\eqref{est19} into account,
 the sequences of $\varphi_m$ and $\nabla\varphi_m$ are bounded in $L^{7 s_0}(Q)$ and in $L^{7s_0/3}(Q)$, respectively. Hence, by means of \eqref{F1} and $\eqref{Fm}_3$, we
have
\begin{align*}
\Vert\nabla F_m'(\varphi_m)\Vert_{L^{7s_0/(\rho+1)}(Q)}\leq k_2\Vert F''(\varphi_m)\nabla\varphi_m\Vert_{L^{7s_0/(\rho+1)}(Q)}\leq C.
\end{align*}
On the other hand, $F_m'(\varphi_m)$
is bounded in $L^2(0,T;L^2(\Omega))$
and hence also in $L^{s_1}(Q)$, where $s_1=\min\{2,7s_0/(\rho+1)\}$.
We therefore deduce that
\begin{align*}
&\Vert F_m'(\varphi_m)\Vert_{L^{s_1}(0,T;W^{1,s_1}(\Omega))}\leq C,\quad s_1:=\min\Big\{2,\frac{7}{\rho+1}s_0\Big\}.
\end{align*}
If $s_1=2$, then by comparison in \eqref{eq2} and using \eqref{basicregapp2} and
elliptic regularity,
we get the desired bound for the sequence of $\varphi_m$ in $L^\infty(0,T;V)\cap L^2(0,T;H^3(\Omega))$.
If $s_1<2$ then, by comparison in \eqref{eq2} and using \eqref{basicregapp2} and
elliptic regularity, we infer
\begin{align*}
\Vert\varphi_m\Vert_{L^\infty(0,T;V)\cap L^{s_1}(0,T;W^{3,s_1}(\Omega))}\leq C.
\end{align*}
Repeating the argument we now have the sequences of $\varphi_m$ and $\nabla\varphi_m$ bounded in $\mathbb{X}_{s_1}$
and in $\mathbb{Y}_{s_1}$, respectively, and hence
$\Vert\nabla F_m'(\varphi_m)\Vert_{L^{7s_1/(\rho+1)}(Q)}\leq C$. Moreover, we know that the sequence of $F_m'(\varphi_m)$
{\color{black}
is bounded in $L^{s_2}(Q)$, where $s_2=\min\{2,7s_1/(\rho+1)\}$.
} This implies
\begin{align*}
&\Vert F_m'(\varphi_m)\Vert_{L^{s_2}(0,T;W^{1,s_2}(\Omega))}\leq C,
\quad {\color{black} s_2:=\min\Big\{2,\Big(\frac{7}{\rho+1}\Big)^2 s_0\Big\}.}
\end{align*}
Again, if {\color{black} $s_2=2$} we get the desired claim; otherwise, by using elliptic regularity we infer that the sequence of $\varphi_m$
is bounded in $\mathbb{X}_{s_2}$
and we repeat the previous argument. By iterating the procedure $k$ times
we get
\begin{align*}
&\Vert F_m'(\varphi_m)\Vert_{L^{s_k}(0,T;W^{1,s_k}(\Omega))}\leq C,\quad
{\color{black} s_k:=\min\Big\{2,\Big(\frac{7}{\rho+1}\Big)^k s_0\Big\}.}
\end{align*}
{\color{black}{\itshape Since $\rho<6$,}} {\color{black}
after a finite number of steps, as soon as we get $s_k=2$, the bootstrap procedure ends yielding
the bound of the sequence of $\varphi_m$ in $L^\infty(0,T;V)\cap L^2(0,T;H^3(\Omega))$
(which cannot be improved since the regularity of $\varphi_m$
is related through \eqref{eq2} to $\mu_m\in L^2(0,T;V)$).}

{\color{black}
As far as the estimates for
the sequences of time derivatives $\varphi_m'$, $\psi_m'$ are concerned, the argument is exactly the same
as for the sequences of time derivatives $\varphi_n'$, $\psi_n'$ of the Faedo-Galerkin approximating solutions
(cf. Step I). Hence, \eqref{est26} still holds for $\varphi_m'$, $\psi_m'$.}
{\color{black}
Finally, the passage to the limit in Problem P$_m$ (notice that $F_m'(\varphi_m)\to F'(\varphi)$ pointwise almost everywhere in $Q$), the proof of the energy inequality \eqref{enineq} for $q\in[1,9)$, the proofs of \eqref{basicreg5}
and of the energy identity for $q\leq 4$ can be carried out along as done at the end of Step I.}\hspace{132mm}$\boxempty$
{\color{black}
The existence of a weak solution without the the optimal regularity $\varphi\in L^2(0,T;H^3(\Omega))$ can still be ensured under a more general assumption on $F$, provided we impose a slight restriction (i.e., $q<7$) on the growth of $p$.}
{\color{black}
More precisely, we have the following}
\begin{cor}\label{cor}
Assume that $F\in C^2(\mathbb{R})$ satisfies
\begin{description}
\item[(F)$_1$] $F''(s)\geq -\lambda_1$,
\item[(F)$_2$] $|F'(s)|\leq\lambda_2 F(s)+\lambda_3$,
\end{description}
{\color{black}
for all $s\in\mathbb{R}$, where $\lambda_1,\lambda_2,\lambda_3$ are some nonnegative constants.
Moreover, assume that $p\in C^{0,1}_{loc}(\mathbb{R})$ satisfies \eqref{assp} with $q\in [1,7)$.
Let $\varphi_0\in V$ and $\psi_0\in H$. Then, for every $T>0$ Problem \eqref{eq1}-\eqref{ics}
admits a weak solution on $[0,T]$ satisfying \eqref{basicreg1}--\eqref{basicreg3}, \eqref{basicreg4}
and the energy inequality \eqref{enineq}.
Finally, if $q\leq 4$, then we have \eqref{basicreg5} and \eqref{enineq} holds with the equality sign.}
\end{cor}
\begin{proof}
{\color{black}
We can follow the Faedo-Galerkin approximation procedure in Step I of the proof of Theorem \ref{existence},
assuming first that $\varphi_0\in D(A)$ in order to control the sequence of $\int_\Omega F(\varphi_{0n})$
in the identity obtained by integrating \eqref{enidapp} in time. Existence of weak solution in the case $\varphi_0\in V$
can then be recovered by means of a density argument. The basic estimates \eqref{basicest1}--\eqref{basicest3}
still hold, as well as the controls \eqref{avcontr}, ensured by (F)$_2$, and \eqref{basicregapp2}.
As far as estimate \eqref{est16} is concerned, this can now be recovered by using (F)$_1$. Indeed,
multiplying \eqref{pbapp2bis} by $\Delta\varphi_n$ in $H$ we get}
{\color{black}
\begin{align*}
\Vert\Delta\varphi_n\Vert^2
&=-(\mu_n,\Delta\varphi_n)+\big(P_n F'(\varphi_n),\Delta\varphi_n\big)\\
&=-(\mu_n,\Delta\varphi_n)-\int_\Omega F''(\varphi_n)|\nabla\varphi_n|^2,
\end{align*}
}
which yields
\begin{align*}
&\Vert\Delta\varphi_n\Vert^2\leq\Vert\mu_n\Vert^2+2\lambda_1\Vert\nabla\varphi_n\Vert^2.
\end{align*}
{\color{black}Estimate \eqref{est16} then follows from this last inequality by using \eqref{basicregapp2}, the first
of \eqref{basicest1} and elliptic regularity.}

{\color{black}
Next, in order to get the control of the sequences of time derivatives $\varphi_n',\psi_n'$
in the space $L^r(0,T;D(A^{-1}))$, for some $r>1$, and in order to pass to the limit in the approximate problem
\eqref{pbapp1}-\eqref{pbapp4} we can still argue as in Step I of the proof of Theorem \ref{existence},
with the difference that now we can only rely in the control given by \eqref{est16}, together
with the following embedding}
\begin{align}
&L^\infty(0,T;V)\cap L^2(0,T;H^2(\Omega))\hookrightarrow L^{4\eta/(\eta-6)}(0,T; L^\eta(\Omega)),\quad\mbox{for } 6\leq\eta\leq\infty.
\label{est44}
\end{align}
{\color{black}Indeed, by using \eqref{est44} with $\eta=42/5$ we can easily see that, since $q\in[1,7)$,
estimates \eqref{est26} and the strong convergence \eqref{strconvp} still hold.}

{\color{black}
As far as the energy inequality \eqref{enineq} is concerned, let us observe that
the sequence of $\varphi_n$ is now bounded in $L^{10}(Q)$ (cf. \eqref{est16} and \eqref{est44}
with $\eta=10$). Hence, on account of \eqref{assp} and of pointwise convergence
we have}
$\sqrt{p(\varphi_n)}\to \sqrt{p(\varphi)}$ strongly in $L^\delta(Q)$, for every $\delta<20/q$.
In particular we have
$\sqrt{p(\varphi_n)}\to \sqrt{p(\varphi)}$ strongly in $L^{5/2}(Q)$, which implies that
\begin{align*}
&\sqrt{p(\varphi_n)}(\mu_n-\psi_n)\rightharpoonup \sqrt{p(\varphi)}(\mu-\psi),\quad\mbox{ in } L^{10/9}(Q).
\end{align*}
Due to $\eqref{basicest2}_3$, this weak convergence also holds in $L^2(Q)$ and
still yields \eqref{liminf} and then \eqref{enineq} as well.

Finally, assume that $q\leq 4$. By arguing as in Step I of the proof of Theorem \ref{existence}
we again deduce \eqref{basicreg5}.
{\color{black}In order to prove that \eqref{enineq} holds with the equality
sign,
let us first observe that from assumption (F) we have $F''(s)\geq-c_\ast$, for some $c_\ast\in\mathbb{R}$,
 and therefore we can write $F$ as
$$F(s)=G_0(s)-c_\ast\frac{s^2}{2},$$
where $G_0\in C^2(\mathbb{R})$ is convex.
Introduce now the functional
$\mathcal{G}_0:H\to\mathbb{R}\cup\{+\infty\}$ given by
$$\mathcal{G}_0(\varphi)=\int_\Omega\Big(\frac{1}{2}|\nabla\varphi|^2+G_0(\varphi)\Big),
\qquad\mbox{if }\varphi\in V\mbox{ and } G_0(\varphi)\in L^1(\Omega),$$
and $\mathcal{G}_0(\varphi)=+\infty$ otherwise.
Then, $\mathcal{G}_0$ is convex and lower semicontinuous on $H$ and
we have (see, e.g., \cite[Proposition 2.8]{Bar})
$$\partial\mathcal{G}_0(\varphi)=-\Delta\varphi+G_0'(\varphi),\qquad\forall
\varphi\in D(\partial\mathcal{G}_0)=D(A).$$
Since $\partial\mathcal{G}_0(\varphi)=-\Delta\varphi+G_0'(\varphi)=\mu+c_\ast\varphi
\in L^2(0,T;V)$, then we can apply
\cite[Proposition 4.2]{CKRS} and write
\begin{align*}
&\langle\varphi_t,\mu\rangle=\langle\varphi_t,\partial\mathcal{G}_0(\varphi)-c_\ast\varphi\rangle
=\frac{d}{dt}\mathcal{G}_0(\varphi)-\frac{c_\ast}{2}\frac{d}{dt}\Vert\varphi\Vert^2
=\frac{d}{dt}\int_\Omega\Big(\frac{1}{2}|\nabla\varphi|^2+F(\varphi)\Big).
\end{align*}
This identity allows to recover \eqref{enid}, and hence \eqref{enineq} with the equality sign,
 by arguing exactly as at the end of Step I
of the proof of Theorem \ref{existence}.} \end{proof}

\color{black}

The next result is concerned with the uniqueness of weak solutions and their continuous dependence with respect to the initial data. In order to prove such a result assumption (F) still suffices, but we need to strengthen (P) as follows
\begin{description}
\item[(P1)]
Let $p\in C^{0,1}_{loc}(\mathbb{R})$  be such that $p\geq 0$ and
\begin{align*}
&|p'(s)|\leq c_5(1+|s|^{q-1}),
\end{align*}
for almost any $s\in\mathbb{R}$, with $1\leq q\leq 4$.
\end{description}
Then we have
\begin{thm}
\label{uniqueness}
Assume that (F) and (P1) are satisfied. Let $\varphi_0\in V$ and $\psi_0\in H$. Then, for every $T>0$
the weak solution to Problem \eqref{eq1}--\eqref{ics} on $[0,T]$ given by Theorem \ref{existence}
is unique. Moreover, let $[\varphi_{0i},\psi_{0i}]\in V\times H$, be two initial data
and $[\varphi_i,\psi_i]$, $i=1,2$ be the corresponding weak solutions. Then, the following
continuous dependence estimate holds
\begin{align*}
&\Vert\varphi_2(t)-\varphi_1(t)\Vert_{V'}+\Vert\psi_2(t)-\psi_1(t)\Vert_{V'}
+\Vert\varphi_2-\varphi_1\Vert_{L^2(0,t;V)}+\Vert\psi_2-\psi_1\Vert_{L^2(0,t;H)}\nonumber\\
&\leq\Lambda(t)\big(\Vert\varphi_{02}-\varphi_{01}\Vert_{V'}+\Vert\psi_{02}-\psi_{01}\Vert_{V'}\big),\quad\forall t\in [0,T],
\end{align*}
where $\Lambda$ is a continuous positive function which depends on the norms of the initial data and on $F$, $p$, $\Omega$ and $T$.
\end{thm}

\begin{oss}
{\upshape
Notice that the restriction $1\leq q\leq 4 $ on the growth of $p$ which is needed to establish the uniqueness is exactly
the same condition which ensures the validity of the energy identity \eqref{enid} which is proven in Theorem \ref{existence}.
}
\end{oss}

\begin{proof}
Let us rewrite the chemical potential $\mu$ and \eqref{wf1}--\eqref{wf2} in the following form
\begin{align}
& \langle\varphi_t,\chi\rangle + (\nabla\mu,\nabla\chi)
+(\mu,\chi) =\big(p(\varphi)\psi-\big(p(\varphi)-1\big)\mu,\chi\big),\label{eq1bis}\\
&\mu=A\varphi+G'(\varphi),\label{eq2bis}\\
&\langle\psi_t,\xi\rangle +(\nabla\psi,\nabla\xi)
+(\psi,\xi) =-\big(\big(p(\varphi)-1\big)\psi+p(\varphi)\mu,\xi\big)\label{eq3bis},
\end{align}
for all $\chi,\xi\in V$, where $G(s):=F(s)-\frac{1}{2}s^2$.

We now write system \eqref{eq1bis}--\eqref{eq3bis}
for two weak solutions $[\varphi_i,\psi_i]$, $i=1,2$, and take the difference of each equation.
Setting $\varphi:=\varphi_2-\varphi_1$, $\psi:=\psi_2-\psi_1$ and $\mu:=\mu_2-\mu_1$, we have
\begin{align}
\nonumber
&\langle\varphi_t,\chi\rangle +(\nabla\mu,\nabla\chi) +(\mu,\chi) \\
&= \big(\big(p(\varphi_2)-p(\varphi_1)\big)(\psi_2-\mu_2)+p(\varphi_1)\psi-\big(p(\varphi_1)-1\big)\mu, \chi\big)\label{eq1diff}\\
&\mu=A\varphi+G'(\varphi_2)-G'(\varphi_1)\label{eq2diff}\\
\nonumber
&\langle\psi_t,\xi\rangle +(\nabla\psi,\nabla\xi) +(\psi,\xi)
\\
&=-\big(\big(p(\varphi_2)-p(\varphi_1)\big)(\psi_2-\mu_2)-\big(p(\varphi_1)-1\big)\psi+p(\varphi_1)\mu,
\xi\big),\label{eq3diff}
\end{align}
for all $\chi,\xi \in V$.
Let us take $ \chi=A^{-1} \varphi$ in \eqref{eq1diff} and $\xi= A^{-1} \psi$ in \eqref{eq3diff} and sum the
resulting identities. Taking also \eqref{eq2diff} into account, we get
\begin{align}
&\frac{1}{2}\frac{d}{dt}\Vert \varphi\Vert_{V'}^2+\Vert \varphi\Vert_V^2+\big(G'(\varphi_2)-G'(\varphi_1),\varphi\big)
+\frac{1}{2}\frac{d}{dt}\Vert\psi\Vert_{V'}^2+\Vert\psi\Vert^2\nonumber\\
&=\Big(\big(p(\varphi_2)-p(\varphi_1)\big)(\psi_2-\mu_2)+p(\varphi_1)\psi-\big(p(\varphi_1)-1\big)\mu,A^{-1}\varphi\Big)\nonumber\\
&+\Big(-\big(p(\varphi_2)-p(\varphi_1)\big)(\psi_2-\mu_2)-\big(p(\varphi_1)-1\big)\psi+p(\varphi_1)\mu,A^{-1}\psi\Big).
\label{diffid1}
\end{align}
We now need to estimate the terms on the right hand side. Observe first that
\begin{align}
&\big(p(\varphi_1)\psi-\big(p(\varphi_1)-1\big)\mu,A^{-1}\varphi\big)
\leq\big(\Vert p(\varphi_1)(\psi-\mu)\Vert_{V'}+\Vert\mu\Vert_{V'}\big)\Vert\varphi\Vert_{V'}.\label{est1}
\end{align}
We have to estimate in $V'$ the term $p(\varphi_1)(\psi-\mu)$.
Let us first estimate $p(\varphi_1)\chi$ in $V$.
By using assumption (P1) we get
\begin{align}
&\Vert p(\varphi_1)\nabla\chi\Vert\leq c\big(1+\Vert\varphi_1\Vert_{L^\infty(\Omega)}^q\big)\Vert\nabla\chi\Vert.
\label{rt1}
\end{align}
Moreover, we have
\begin{align}
&\Vert p'(\varphi_1)\nabla\varphi_1\chi\Vert\leq\Vert p'(\varphi_1)\nabla\varphi_1\Vert_{L^3(\Omega)}\Vert\chi\Vert_{L^6(\Omega)}
\leq \Vert p'(\varphi_1)\nabla\varphi_1\Vert_{L^3(\Omega)}\Vert\chi\Vert_V.\label{rt2}
\end{align}
However, $\nabla\varphi_1\in L^\infty(0,T;H)\cap L^2(0,T;H^2(\Omega))\hookrightarrow L^8(0,T;L^3(\Omega))$.
On the other hand, $\varphi_1\in L^\infty(0,T;V)\cap L^2(0,T;H^3(\Omega))\hookrightarrow L^8(0,T;L^\infty(\Omega))$
(cf. \eqref{estro} with $\theta=\infty$).
Thus, thanks to assumption (P1), we also have $p'(\varphi_1)\in L^{8/(q-1)}(0,T;L^\infty(\Omega))$.
Hence, we find
\begin{align}
& p'(\varphi_1)\nabla\varphi_1\in L^{8/q}(0,T;L^3(\Omega)).\label{rt3}
\end{align}
Moreover, observe that
\begin{align}
&\Vert p(\varphi_1)\chi\Vert\leq c\Vert p(\varphi_1)\Vert_{L^3(\Omega)}\Vert\chi\Vert_V,\label{rt4}
\end{align}
and
\begin{align}
&\Vert p(\varphi_1)\Vert_{L^3(\Omega)}\leq c\big(1+\Vert\varphi_1\Vert_{L^{3q}(\Omega)}^q\big).\label{rt5}
\end{align}
Observing that $\varphi_1\in L^\infty(0,T;V)\cap L^2(0,T;H^3(\Omega))\hookrightarrow L^{8q/(q-2)}(0,T;L^{3q}(\Omega))$
(cf. \eqref{estro}), then we have
\begin{align}
&p(\varphi_1)\in L^{8/(q-2)}(0,T;L^3(\Omega)).\label{rt6}
\end{align}
By collecting \eqref{rt1}--\eqref{rt6} we get
\begin{align*}
\Vert p(\varphi_1)\chi\Vert_V\leq\alpha_1(t)\Vert\chi\Vert_V,
\end{align*}
where the function $\alpha_1$ is given by
\begin{align*}
&\alpha_1(t):= c\big(\Vert p\big(\varphi_1(t)\big)\Vert_{L^3(\Omega)}+\Vert\varphi_1(t)\Vert_{L^\infty(\Omega)}^q
+\Vert p'\big(\varphi_1(t)\big)\nabla\varphi_1(t)\Vert_{L^3(\Omega)}+1\big),
\end{align*}
and, since $q\leq 4$, we have $\alpha_1\in L^2(0,T)$.
Therefore, we obtain
\begin{align}
&|\langle p(\varphi_1)(\psi-\mu),\chi\rangle|\leq|\big((\psi-\mu),p(\varphi_1)\chi\big)|\leq
\alpha_1(t)\Vert\psi-\mu\Vert_{V'}\Vert\chi\Vert_{V},\nonumber
\end{align}
which yields
\begin{align}
&\Vert p(\varphi_1)(\psi-\mu)\Vert_{V'}\leq \alpha_1(t)\Vert\psi-\mu\Vert_{V'}.\label{est2}
\end{align}
By combining \eqref{est1} with \eqref{est2} we deduce
\begin{align}
&\big(p(\varphi_1)\psi-\big(p(\varphi_1)-1\big)\mu,A^{-1}\varphi\big)
\leq \alpha_1(t)\big(\Vert\psi\Vert_{V'}+\Vert\mu\Vert_{V'}\big)\Vert\varphi\Vert_{V'}.
\label{est3}
\end{align}
For the estimate of $\mu$ in $V'$, by means of assumption (F) and using the continuous embedding
$L^{6/5}(\Omega)\hookrightarrow V'$, it is easy to see that
\begin{align}
\Vert\mu\Vert_{V'}
&\leq\Vert\varphi\Vert_V+\Vert G'(\varphi_2)-G'(\varphi_1)\Vert_{V'}\nonumber\\
&\leq\Vert\varphi\Vert_V+ c\big(1+
\Vert\varphi_1\Vert_{L^{3(\rho-2)/2}(\Omega)}^{\rho-2}
+\Vert\varphi_2\Vert_{L^{3(\rho-2)/2}(\Omega)}^{\rho-2}\big)\Vert\varphi\Vert_{L^6(\Omega)}\nonumber\\
&\leq c\big(1+\Vert\varphi_1\Vert_V^{\rho-2}+\Vert\varphi_2\Vert_V^{\rho-2}\big)\Vert\varphi\Vert_V
\leq \Gamma\Vert\varphi\Vert_V,\label{est4}
\end{align}
since $3(\rho-2)/2\leq 6$, {\color{black}being $\rho< 6$}.
In the last inequality we have used $\eqref{basicreg1}_1$.
In \eqref{est4} and also in the estimates below, $\Gamma$ denotes
a positive constant that depends on the norms of the initial
data of the two solutions, i.e., $\Gamma=\Gamma\big(\Vert\varphi_{01}\Vert_V,\Vert\varphi_{02}\Vert_V,\Vert\psi_{01}\Vert,\Vert\psi_{02}\Vert\big)$
(of course, $\Gamma$ depends also on $F$ and $\Omega$).
{\color{black} The value of $\Gamma$ may change even within the same line}.
From \eqref{est3} and \eqref{est4} we get
\begin{align}
|\big(p(\varphi_1)\psi-\big(p(\varphi_1)-1\big)\mu,A^{-1}\varphi\big)|
&\leq \alpha_1(t)\Gamma\big(\Vert\psi\Vert_{V'}+\Vert\varphi\Vert_{V}\big)\Vert\varphi\Vert_{V'}\nonumber\\
&\leq\frac{1}{10}\Vert\varphi\Vert_V^2
+\Gamma\alpha_1^2(t)\big(\Vert\psi\Vert_{V'}^2+\Vert\varphi\Vert_{V'}^2\big).\label{est6}
\end{align}
The next term on the right hand side of \eqref{diffid1} to be estimated is the following
\begin{align}
&|\big(\big(p(\varphi_2)-p(\varphi_1)\big)(\psi_2-\mu_2),A^{-1}\varphi\big)|\leq
\Vert(p(\varphi_2)-p(\varphi_1)\big)(\psi_2-\mu_2)\Vert_{V'}\Vert\varphi\Vert_{V'}.\label{est5}
\end{align}
Let us first control the term $(\big(p(\varphi_2)-p(\varphi_1)\big)(\psi_2-\mu_2)$ in $V'$. We have, for every $\chi\in V$,
\begin{align}
\big|\big\langle\big(p(\varphi_1)-p(\varphi_2)\big)(\psi_2-\mu_2),\chi\big\rangle\big)\big|
&\leq\Vert p(\varphi_1)-p(\varphi_2)\Vert\Vert\psi_2-\mu_2\Vert_{L^3}\Vert\chi\Vert_{L^6}\nonumber\\
&\leq c\Vert p(\varphi_1)-p(\varphi_2)\Vert\Vert\psi_2-\mu_2\Vert_{L^3}\Vert\chi\Vert_{V}.\label{est22}
\end{align}
On the other hand, thanks to (P1), we obtain
\begin{align}
\Vert p(\varphi_2)-p(\varphi_1)\Vert
&\leq c\big(1+\Vert\varphi_1\Vert_{L^\infty(\Omega)}^{q-1}
+\Vert\varphi_2\Vert_{L^\infty(\Omega)}^{q-1}\big)\Vert\varphi\Vert\nonumber\\
&\leq c\big(1+\Vert\varphi_1\Vert_{L^\infty(\Omega)}^{q-1}
+\Vert\varphi_2\Vert_{L^\infty(\Omega)}^{q-1}\big)\Vert\varphi\Vert_{V'}^{1/2}\Vert\varphi\Vert_V^{1/2}.\label{est23}
\end{align}
Moreover, by using \eqref{GN}
{\color{black} and the interpolation inequality
$\Vert\mu_2\Vert\leq \Vert\mu_2\Vert_{V'}^{1/2}\Vert\mu_2\Vert_V^{1/2}$,} we get
\begin{align}
&\Vert\mu_2\Vert_{L^3(\Omega)}\leq c\Vert\mu_2\Vert^{1/2}\Vert\mu_2\Vert_V^{1/2}\leq c\Vert\mu_2\Vert_{V'}^{1/4}\Vert\mu_2\Vert_V^{3/4}
{\color{black}\leq\Gamma
\Vert\mu_2\Vert_V^{3/4}},\label{est25}
\end{align}
where {\color{black} in the last estimate we have exploited the inequality
$\Vert\mu_2\Vert_{V'}\leq \Gamma(1+\Vert\varphi_2\Vert_V)\leq\Gamma$ which} can be
deduced by arguing as in \eqref{est4}.
 Hence, from \eqref{est22}--\eqref{est25} we infer
 \begin{align}
 &\Vert(p(\varphi_2)-p(\varphi_1)\big)(\psi_2-\mu_2)\Vert_{V'}
 \leq \alpha_2(t)
\Vert\varphi\Vert_{V'}^{1/2}\Vert\varphi\Vert_V^{1/2},
 \label{est45}
 \end{align}
 where
 \begin{align}
 &\alpha_2(t):=c\big(1+\Vert\varphi_1(t)\Vert_{L^\infty(\Omega)}^{q-1}
+\Vert\varphi_2(t)\Vert_{L^\infty(\Omega)}^{q-1}\big)\big(\Vert\psi_2(t)\Vert_{L^3(\Omega)}
{\color{black}+\Gamma
\Vert\mu_2(t)\Vert_V^{3/4}}\big).
\label{defalpha2}
 \end{align}
 Observe that $\alpha_2\in L^{4/3}(0,T)$ since $q\leq 4$.
 Indeed, both factors in \eqref{defalpha2} are in $L^{8/3}(0,T)$, recalling that $\varphi_1,\varphi_2\in
 L^\infty(0,T;V)\cap L^2(0,T;H^3(\Omega))\hookrightarrow L^8(0,T;L^\infty(\Omega))$ and properties
 \eqref{basicreg1}--\eqref{basicreg3} {\color{black}(in particular we have $\psi_2\in L^{10/3}(Q)$).}
Hence, from \eqref{est5} we get
\begin{align}
|\big(\big(p(\varphi_2)-p(\varphi_1)\big)(\psi_2-\mu_2),A^{-1}\varphi\big)|
&\leq
\alpha_2(t)\Vert\varphi\Vert_V^{1/2}\Vert\varphi\Vert_{V'}^{3/2}\nonumber\\
&\leq\frac{1}{10}\Vert\varphi\Vert_V^2+\alpha_2^{4/3}(t)\Vert\varphi\Vert_{V'}^2.
\label{est7}
\end{align}
We now estimate the following term (cf. the right hand side of \eqref{diffid1})
\begin{align}
\big|\big(-\big(p(\varphi_2)-p(\varphi_1)\big)(\psi_2-\mu_2),A^{-1}\psi\big)\big|
&{\color{black}\leq
\Vert(p(\varphi_2)-p(\varphi_1)\big)(\psi_2-\mu_2)\Vert_{V'}\Vert A^{-1}\psi\Vert_{V}}\nonumber\\
&{\color{black}=}
\Vert(p(\varphi_2)-p(\varphi_1)\big)(\psi_2-\mu_2)\Vert_{V'}\Vert\psi\Vert_{V'}\nonumber\\
&\leq
\alpha_2(t)
\Vert\varphi\Vert_{V'}^{1/2}\Vert\varphi\Vert_V^{1/2}\Vert\psi\Vert_{V'}\nonumber\\
&
\leq\frac{1}{10}\Vert\varphi\Vert_V^2+\alpha_2^{4/3}(t)\Vert\varphi\Vert_{V'}^{2/3}\Vert\psi\Vert_{V'}^{4/3}\nonumber\\
&\leq\frac{1}{10}\Vert\varphi\Vert_V^2+\alpha_2^{4/3}(t)\big(\Vert\varphi\Vert_{V'}^2+\Vert\psi\Vert_{V'}^2\big),
\label{est9}
\end{align}
{\color{black} where, in the third inequality, \eqref{est45} has been used.}
We now estimate the last term on the right hand side of \eqref{diffid1}
\begin{align}
\big|\big(-\big(p(\varphi_1)-1\big)\psi+p(\varphi_1)\mu,A^{-1}\psi\big)\big|
&\leq
\big(\Vert p(\varphi_1)(\psi-\mu)\Vert_{V'}+\Vert\psi\Vert_{V'}\big)\Vert\psi\Vert_{V'}\nonumber\\
&\leq
\big(\alpha_1\Vert\psi-\mu\Vert_{V'}
+\Vert\psi\Vert_{V'}\big)\Vert\psi\Vert_{V'}\nonumber\\
&\leq {\color{black}(1+\alpha_1)}\Vert\psi\Vert_{V'}^2+\alpha_1\Gamma\Vert\varphi\Vert_V\Vert\psi\Vert_{V'}\nonumber\\
&\leq \frac{1}{10}\Vert\varphi\Vert_V^2+\Gamma{\color{black}(1+\alpha_1^2)}\Vert\psi\Vert_{V'}^2,\label{est8}
\end{align}
{\color{black} where we have used \eqref{est2} in the second inequality and \eqref{est4} in the third inequality.}

Moreover, setting {\color{black}$\hat{\beta}:=\alpha+1-c_1$}, we have
\begin{align}
&\big(G'(\varphi_2)-G'(\varphi_1),\varphi\big)\geq-\hat{\beta}\Vert\varphi\Vert^2\geq-\frac{1}{10}\Vert\varphi\Vert_V^2
-c\Vert\varphi\Vert_{V'}^2.\label{est10}
\end{align}

Finally, plugging estimates \eqref{est6} and \eqref{est7}--\eqref{est10} into \eqref{diffid1}
yields the following differential inequality
\begin{align}
&\frac{d}{dt}\Big(\Vert\varphi\Vert_{V'}^2+\Vert\psi\Vert_{V'}^2\Big)+\Vert\varphi\Vert_V^2+\Vert\psi\Vert^2
\leq\hat{\gamma}\Big(\Vert\varphi\Vert_{V'}^2+\Vert\psi\Vert_{V'}^2\Big),\label{diffineq1}
\end{align}
where
\begin{align*}
&\hat{\gamma}:=\Gamma\big(\alpha_1^2+\alpha_2^{4/3}+1\big)\in L^1(0,T).
\end{align*}
An application of Gronwall's inequality to \eqref{diffineq1} ends the proof.
\end{proof}

\section{Strong solutions and the global attractor}
\label{sec:strong}

Here we establish a regularity result for Problem \eqref{eq1}--\eqref{ics} that holds under the same
condition on $p$ which ensures uniqueness (cf. (P1)). This result will be used to deduce
some uniform in time higher-order estimates which will be crucial in order to prove
the existence of the global attractor.

\begin{thm}
Suppose (F) and (P1) hold. Let $\varphi_0\in H^3(\Omega)$ and $\psi_0\in V$.
Then, for every $T>0$, the solution $[\varphi,\psi]$ to Problem \eqref{eq1}--\eqref{ics} on $[0,T]$ given
by Theorem \ref{existence} satisfies
\begin{align*}
&\varphi\in L^\infty(0,T;H^3(\Omega)),\quad\varphi_t\in L^2(0,T;V),\\
&\mu\in L^\infty(0,T;V),\\
&\psi\in L^\infty(0,T;V),\quad\psi_t\in L^2(0,T;H).
\end{align*}

\end{thm}

\begin{proof}
The proof is carried out by deducing formally some higher order identities and estimates
which can be justified rigorously by means of a suitable approximation
procedure (see the proof of Theorem \ref{existence}).

Testing \eqref{eq1} by $\mu_t$ in $H$ and using \eqref{eq2}, we find
\begin{align}
&\frac{1}{2}\frac{d}{dt}\Vert\nabla\mu\Vert^2+\Vert\nabla\varphi_t\Vert^2+\int_\Omega F''(\varphi)\varphi_t^2
=\big(p(\varphi)(\psi-\mu),\mu_t\big),\nonumber
\end{align}
whence
\begin{align}
&\frac{1}{2}\frac{d}{dt}\Vert\nabla\mu\Vert^2+\Vert\nabla\varphi_t\Vert^2+\int_\Omega F''(\varphi)\varphi_t^2
+\frac{1}{2}\frac{d}{dt}\int_\Omega p(\varphi)\mu^2=\frac{1}{2}\int_\Omega p'(\varphi)\varphi_t \mu^2+\big(p(\varphi)\psi,\mu_t\big).
\label{diffid2}
\end{align}
Test now \eqref{eq3} by $\psi_t$ in $H$ to get
\begin{align}
&\Vert\psi_t\Vert^2=-\frac{1}{2}\frac{d}{dt}\Vert\nabla\psi\Vert^2-\frac{1}{2}\frac{d}{dt}\int_\Omega p(\varphi)\psi^2
+\frac{1}{2}\int_\Omega p'(\varphi)\varphi_t\psi^2+\big(p(\varphi)\mu,\psi_t\big).
\label{diffid3}
\end{align}
Summing \eqref{diffid2} with \eqref{diffid3} we obtain
\begin{align}
&\frac{1}{2}\frac{d}{dt}\Vert\nabla\mu\Vert^2+\Vert\nabla\varphi_t\Vert^2+\int_\Omega F''(\varphi)\varphi_t^2
+\frac{1}{2}\frac{d}{dt}\int_\Omega p(\varphi)\mu^2\nonumber\\
&+\Vert\psi_t\Vert^2+\frac{1}{2}\frac{d}{dt}\Vert\nabla\psi\Vert^2
+\frac{1}{2}\frac{d}{dt}\int_\Omega p(\varphi)\psi^2\nonumber\\
&=\frac{1}{2}\int_\Omega p'(\varphi)\varphi_t \mu^2+\frac{d}{dt}\int_\Omega p(\varphi)\psi\mu-\int_\Omega p'(\varphi)\varphi_t\psi\mu
+\frac{1}{2}\int_\Omega p'(\varphi)\varphi_t\psi^2,\nonumber
\end{align}
so that
\begin{align}
&\frac{1}{2}\frac{d}{dt}\Big(\Vert\nabla\mu\Vert^2+\Vert\nabla\psi\Vert^2+\int_\Omega p(\varphi)(\mu-\psi)^2\Big)
+\Vert\nabla\varphi_t\Vert^2+\int_\Omega F''(\varphi)\varphi_t^2+\Vert\psi_t\Vert^2\nonumber\\
&=\frac{1}{2}\int_\Omega p'(\varphi)\varphi_t(\mu-\psi)^2.\label{diffid4}
\end{align}
Observe now that
\begin{align}
&\Big|\frac{1}{2}\int_\Omega p'(\varphi)\varphi_t(\mu-\psi)^2\Big|
\leq\frac{1}{2}\Vert p'(\varphi)\Vert\Vert\varphi_t\Vert_{L^6(\Omega)}\Vert\mu-\psi\Vert_{L^6(\Omega)}^2
\leq c\Vert p'(\varphi)\Vert\Vert\varphi_t\Vert_V\Vert\mu-\psi\Vert_V^2.
\label{est12}
\end{align}
Moreover, we have (using \eqref{eq1} and \eqref{bcs})
\begin{align}
\Vert\varphi_t\Vert_V
&\leq (1+c_{\Omega})\Vert\nabla\varphi_t\Vert+|\Omega|^{1/2}|\overline{\varphi}_t|\nonumber\\
&\leq(1+c_{\Omega})\Vert\nabla\varphi_t\Vert+\frac{1}{|\Omega|^{1/2}}\Big|\int_\Omega p(\varphi)(\mu-\psi)\Big|\nonumber\\
&\leq(1+c_{\Omega})\Vert\nabla\varphi_t\Vert+\frac{1}{|\Omega|^{1/2}}\Vert p(\varphi)\Vert_{L^{6/5}(\Omega)}\Vert\mu-\psi\Vert_{L^6(\Omega)},
\label{est13}
\end{align}
where $c_\Omega$ is the constant appearing in the Poincar\'{e}-Wirtinger inequality.
Hence, by combining \eqref{est12} with \eqref{est13}, we get, \color{black} appying, in particular, the Young inequality with exponents $4$ and $4/3$ in the last line, \color{black}
\begin{align}
&\Big|\frac{1}{2}\int_\Omega p'(\varphi)\varphi_t(\mu-\psi)^2\Big|
\leq c\Vert p'(\varphi)\Vert\big(\Vert\nabla\varphi_t\Vert+\Vert p(\varphi)\Vert_{L^{6/5}(\Omega)}\Vert\mu-\psi\Vert_V\big)
\big(\Vert\mu\Vert_V^2+\Vert\psi\Vert_V^2\big)\nonumber\\
&{\color{black}
\leq \frac{1}{2}\Vert\nabla\varphi_t\Vert^2+c\Vert p'(\varphi)\Vert^2\big(\Vert\mu\Vert_V^4+\Vert\psi\Vert_V^4\big)
+c\Vert p'(\varphi)\Vert\Vert p(\varphi)\Vert_{L^{6/5}(\Omega)}\big(\Vert\mu\Vert_V^3+\Vert\psi\Vert_V^3\big)
}
\nonumber\\
&\leq \frac{1}{2}\Vert\nabla\varphi_t\Vert^2+c\big(1+\Vert p'(\varphi)\Vert^2\big)\big(\Vert\mu\Vert_V^4+\Vert\psi\Vert_V^4\big)
+c\Vert p'(\varphi)\Vert^4\Vert p(\varphi)\Vert_{L^{6/5}(\Omega)}^4.
\label{est28}
\end{align}
Thanks to (P1) and to $\eqref{basicreg1}_1$ we can see that $p'(\varphi)$
is controlled in $L^\infty(0,T;H)$.
Moreover, we know that $\varphi$ is bounded in $L^{18}(0,T;L^{54/5}(\Omega))$ (cf. \eqref{est39})
and $\varphi$ is also bounded in $L^{4q}(0,T;L^{6q/5}(\Omega))$ since $q\leq 4$,
Thanks to this bound, assumption (P1) entails that $p(\varphi)$ is controlled in $L^4(0,T; L^{6/5}(\Omega))$.
Thus we have
\begin{align}
&\Vert p'(\varphi)\Vert_{L^\infty(0,T;H)}\leq \Gamma,\quad \Vert p(\varphi)\Vert_{L^4(0,T; L^{6/5}(\Omega))}\leq \Gamma,
\label{est29}
\end{align}
where henceforth $\Gamma=\Gamma\big(\Vert\varphi_0\Vert_V,\Vert\psi_0\Vert\big)$ will denote
a positive constant that depends on the norms of the initial data (and on $F$, $p$, $\Omega$).
Furthermore, we have
\begin{align}
&\Vert\mu\Vert_V\leq(1+c_\Omega)\Vert\nabla\mu\Vert+|\Omega|^{1/2}|\overline{\mu}|\leq(1+c_\Omega)\Vert\nabla\mu\Vert+\Gamma,
\label{est30}\\
&\Vert\psi\Vert_V\leq \Vert\nabla\psi\Vert+\Gamma\label{est31}.
\end{align}

Plugging estimate \eqref{est28}  into \eqref{diffid4} and using \eqref{est29}, \eqref{est30}, \eqref{est31} and
\eqref{F2}, we get
\begin{align}
&\frac{1}{2}\frac{d}{dt}\Big(\Vert\nabla\mu\Vert^2+\Vert\nabla\psi\Vert^2+\int_\Omega p(\varphi)(\mu-\psi)^2\Big)
+\frac{1}{2}\Vert\nabla\varphi_t\Vert^2+\Vert\psi_t\Vert^2\leq c_3\Vert\varphi_t\Vert^2\nonumber\\
&+
\Gamma\big(\Vert\mu\Vert_V^2\Vert\nabla\mu\Vert^2+\Vert\psi\Vert_V^2\Vert\nabla\psi\Vert^2\big)
+\Gamma\big(\Vert\mu\Vert_V^2+\Vert\psi\Vert_V^2+\Vert p(\varphi)\Vert_{L^{6/5}(\Omega)}^4\big).
\label{diffineq2}
\end{align}
We now need an estimate for the $L^2$-norm of $\varphi_t$ in \eqref{diffineq2}. This can be obtained by testing
\eqref{eq1} by $\varphi_t$ in $H$, integrating by parts in $\Omega$
and using \eqref{eq2}. This yields
\begin{align}
&\Vert\varphi_t\Vert^2=(\mu,\Delta\varphi_t)+\big(p(\varphi)(\psi-\mu),\varphi_t\big)\nonumber\\
&=-\frac{1}{2}\frac{d}{dt}\Vert\Delta\varphi\Vert^2-\int_\Omega F''(\varphi)\nabla\varphi\cdot\nabla\varphi_t
+\big(p(\varphi)(\psi-\mu),\varphi_t\big).\nonumber
\end{align}
Hence, we have
\begin{align}
&\frac{1}{2}\frac{d}{dt}\Vert\Delta\varphi\Vert^2
+\frac{1}{2}\Vert\varphi_t\Vert^2\leq\Big|\int_\Omega F''(\varphi)\nabla\varphi\cdot\nabla\varphi_t\Big|
+\frac{1}{2}\Vert p(\varphi)\Vert_{L^3(\Omega)}^2\Vert\mu-\psi\Vert_{L^6(\Omega)}^2
\nonumber\\
&
\leq\Vert F''(\varphi)\Vert_{L^{7/2}(\Omega)}\Vert\nabla\varphi\Vert_{L^{14/3}(\Omega)}\Vert\nabla\varphi_t\Vert
+c\Vert p(\varphi)\Vert_{L^3(\Omega)}^2\big(\Vert\mu\Vert_V^2+\Vert\psi\Vert_V^2\big)\nonumber\\
&\leq \frac{1}{8c_3}\Vert\nabla\varphi_t\Vert^2+
c\Vert F''(\varphi)\Vert_{L^{7/2}(\Omega)}^2\Vert\nabla\varphi\Vert_{L^{14/3}(\Omega)}^2
+c\Vert p(\varphi)\Vert_{L^3(\Omega)}^2\big(\Vert\mu\Vert_V^2+\Vert\psi\Vert_V^2\big).
\label{diffineq3}
\end{align}
Recalling that $\varphi$ is bounded in $L^{14}(Q)$ (cf \eqref{est17} and \eqref{estro} with $\theta=14$),
(F) implies that $F''(\varphi)$ is bounded in $L^{7/2}(Q)$
{\color{black}(note that $\rho<6$)}.
Moreover, $\nabla\varphi$ is bounded in $L^{14/3}(Q)$ (cf. \eqref{est17} and \eqref{est19} with $s=2$).
Therefore the second term on the right hand side of
the last inequality in \eqref{diffineq3} is bounded in $L^1(0,T)$.

Furthermore, $\varphi$ is also bounded in $L^8(0,T;L^\infty(\Omega))$
(cf. \eqref{est17} and \eqref{estro} with $\theta=\infty$) and, being $q\leq 4$, (P1)
implies that $p(\varphi)$ is bounded in $L^2(0,T;L^3(\Omega))$.

By combining \eqref{diffineq2} with \eqref{diffineq3}, also on account of \eqref{est30} and \eqref{est31},
we obtain the following differential inequality
\begin{align}
&\frac{1}{2}\frac{d}{dt}\Big(\Vert\nabla\mu\Vert^2+\Vert\nabla\psi\Vert^2+2c_3\Vert\Delta\varphi\Vert^2+\int_\Omega p(\varphi)(\mu-\psi)^2\Big)
+\frac{1}{4}\Vert\nabla\varphi_t\Vert^2+\Vert\psi_t\Vert^2\nonumber\\
&\leq \sigma_1\big(\Vert\nabla\mu\Vert^2+\Vert\nabla\psi\Vert^2\big)+\sigma_2,\label{diffineq4}
\end{align}
where
\begin{align}
&\sigma_1:=c\Vert p(\varphi)\Vert_{L^3(\Omega)}^2,\quad \sigma_2:=c\Vert F''(\varphi)\Vert_{L^{7/2}(\Omega)}^2\Vert\nabla\varphi\Vert_{L^{14/3}(\Omega)}^2+\Gamma\Vert p(\varphi)\Vert_{L^3(\Omega)}^2.
\label{sigmai}
\end{align}
Notice that
$$\Vert\sigma_1\Vert_{L^1(0,T)}\leq \Gamma,\quad \Vert\sigma_2\Vert_{L^1(0,T)}\leq\Gamma.$$
Using Gronwall's lemma and recalling the assumptions on the initial data
(in particular, $\varphi_0\in H^3(\Omega)$ implies that $\mu(0)\in V$) from \eqref{diffineq4} we get
that $\nabla\mu$ and $\Delta\varphi$ belong to $L^\infty(0,T;H)$, $\psi\in L^\infty(0,T;V)$,
$\nabla\varphi_t$ and $\psi_t$ belong to $L^2(0,T;H)$.
Also, thanks to (F), we have that $F'(\varphi)\in L^\infty(0,T;H)$. Therefore $\mu\in L^\infty(0,T;H)$ so that
\begin{align}
&\mu\in L^\infty(0,T;V).\label{bound1}
\end{align}
Moreover, due to elliptic regularity result for the homogeneous Neumann problem, we deduce
$\varphi\in L^\infty(0,T;H^2(\Omega))$. From this property and \eqref{bound1} we infer
we have also
\begin{align}
&\varphi\in L^\infty(0,T;H^3(\Omega)).\label{bound2}
\end{align}
Indeed, since $\varphi\in L^\infty(0,T;H^2(\Omega))$, we have $F'(\varphi)\in L^\infty(0,T;V)$.
From \eqref{bound1} we then get $\Delta\varphi\in L^\infty(0,T;V)$ and \eqref{bound2} follows by
standard elliptic regularity.

Finally, as far as $\varphi_t$ is concerned, by integrating \eqref{diffineq3} in time between $0$ and $t$ we get
$\varphi_t\in L^2(0,T;H)$ and this bound together with the bound for $\nabla\varphi_t$ deduced above
imply $\varphi_t\in L^2(0,T;V)$. \end{proof}


We now show that \eqref{eq1}--\eqref{bcs} define a dynamical system on a suitable phase space.

Let $M>0$ be given. Set
\begin{align*}
&\mathcal{W}_M:=\{w=[\varphi,\psi]\in V\times H:\:\:\mathcal{E}(w)\leq M\}.
\end{align*}
and endow $\mathcal{W}_M$ with the metric
\begin{align*}
&\textbf{d}_{\mathcal{W}_M}(w_2,w_1):=
\Vert\varphi_2-\varphi_1\Vert_V+\Vert\psi_2-\psi_1\Vert,\quad\forall w_i:=[\varphi_i,\psi_i]\in \mathcal{W}_M,\quad i=1,2,
\end{align*}
so that it is a complete metric space.
As a consequence of Theorem \ref{existence} and Theorem \ref{uniqueness}, assuming that
(F) and (P1) are satisfied, we can define a semigroup $\{S_M(t)\}_{t\geq 0}$ of {\itshape closed} operators on $\mathcal{W}_M$ (cf. \cite{PZ}) by setting
$$
[\varphi(t),\psi(t)]=S_M(t)[\varphi_0,\psi_0], \quad \forall t\geq 0,
$$
where $[\varphi,\psi]$ is the unique (weak) solution to Problem \eqref{eq1}--\eqref{ics}.

Notice that we have the total mass constraint
\begin{align*}
\big|\overline{\varphi(t)}+\overline{\psi(t)}\big|
=|\overline{\varphi}_0+\overline{\psi}_0|\leq Q(M),\quad\forall t\geq 0,
\end{align*}
where henceforth by $Q=Q(M)$ we denote a nonnegative continuous monotone increasing function of $M$
(which may also depend on $F$, $p$ and $\Omega$). Such function may change even within the same line.

\begin{thm}
Let (F) and (P1) be satisfied. Then the dynamical system $(\mathcal{W}_M,\{S_M(t)\}_{t\geq 0})$
possesses the global attractor.
\end{thm}
\begin{proof}
{\color{black}
The proof is carried out by showing the existence of a compact (in $\mathcal{W}_M$) absorbing set $\mathcal{B}_M$
for the semigroup $\{S_M(t)\}_{t\geq 0}$. This fact will allow us to apply a general result on the existence of global attractors
for semigroup of closed operators proven in \cite{PZ}.

Let us first} write \eqref{diffineq4} in the form
\begin{align}
&\frac{d\Phi}{dt}
+\frac{1}{4}\Vert\nabla\varphi_t\Vert^2+\Vert\psi_t\Vert^2
\leq \sigma_1\Phi
+\sigma_2,\label{diffineq4bis}
\end{align}
where
\begin{align*}
&\Phi:=\frac{1}{2}\Vert\nabla\mu\Vert^2+\frac{1}{2}\Vert\nabla\psi\Vert^2+c_3\Vert\Delta\varphi\Vert^2+\frac{1}{2}\int_\Omega p(\varphi)(\mu-\psi)^2.
\end{align*}
and $\sigma_1$ and $\sigma_2$ are defined as in \eqref{sigmai}.
Notice that, since $\Gamma=\Gamma\big(\Vert\varphi_0\Vert_V,\Vert\psi_0\Vert\big)$ and since $[\varphi_0,\psi_0]\in \mathcal{W}_M$,
then the constant $\Gamma$ that
bounds the $L^1-$norm of $\sigma_2$ will depend only on $M$.


Integrating the energy identity \eqref{enid} between $t$ and $t+1$, we get, for all $t\geq 0$.
\begin{align}
&\int_t^{t+1}\Vert\nabla\mu\Vert^2 d\tau\leq M,\quad\int_t^{t+1}\Vert\nabla\psi\Vert^2 d\tau\leq M,
\quad\int_t^{t+1}\int_\Omega p(\varphi)(\mu-\psi)^2\leq M.
\label{est34}
\end{align}
Recalling that $q\leq 4$, we deduce from (P1) that
\begin{align}
&\int_t^{t+1}\sigma_1(\tau)d\tau\leq c\big(1+\Vert\varphi\Vert_{L^{2q}(t,t+1;L^{3q}(\Omega))}^{2q}\big)\nonumber\\
&\leq c\big(1+\Vert\varphi\Vert_{L^\infty(t,t+1;V)}^{2q}+\Vert\varphi\Vert_{L^2(t,t+1;H^3(\Omega))}^{2q}\big)
\leq Q(M).\label{est32}
\end{align}
Moreover, on account of (F) and (P1), we obtain
\begin{align}
&\int_t^{t+1}\sigma_2(\tau)d\tau
\leq c\Vert F''(\varphi)\Vert_{L^{7/2}(t,t+1;L^{7/2}(\Omega))}^2\Vert\nabla\varphi\Vert_{L^{14/3}(t,t+1;L^{14/3}(\Omega))}^2
+Q(M)\nonumber\\
&\leq c\big(1+\Vert\varphi\Vert_{L^{7(\rho-2)/2}(t,t+1;L^{7(\rho-2)/2}(\Omega)}^{2(\rho-2)}\big)
\Vert\nabla\varphi\Vert_{L^{14/3}(t,t+1;L^{14/3}(\Omega))}^2
+Q(M)\nonumber\\
&\leq c\big(1+\Vert\varphi\Vert_{L^\infty(t,t+1;V)}^{2(\rho-2)}+\Vert\varphi\Vert_{L^2(t,t+1;H^3(\Omega))}^{2(\rho-2)}\big)
\big(\Vert\nabla\varphi\Vert_{L^\infty(t,t+1;H)}^2+\Vert\nabla\varphi\Vert_{L^2(t,t+1;H^2(\Omega))}^2\big)\nonumber\\
&+Q(M)\leq c\big(1+\Vert\varphi\Vert_{L^\infty(t,t+1;V)}^{2(\rho-1)}+\Vert\varphi\Vert_{L^2(t,t+1;H^3(\Omega))}^{2(\rho-1)}\big)
+Q(M)\leq Q(M).\label{est33}
\end{align}

In \eqref{est32} and \eqref{est33} we have used the fact that the $L^2(t,t+1;H^3(\Omega))-$norm of $\varphi$ can be controlled, uniformly in time, in terms of $\Vert\varphi_0\Vert_V,\Vert\psi_0\Vert$ and hence of $M$, when {\color{black}$4<\rho<6$.}
 Indeed, we can use the iteration argument outlined in the proof of Theorem \ref{existence} (cf. Step II; if $\rho=4$ no iteration is needed).

Therefore, we have (see \eqref{est34})
\begin{align}
&\int_t^{t+1}\Phi(\tau) d\tau\leq\frac{3M}{2}+c_3\int_t^{t+1}\Vert\Delta\varphi(\tau)\Vert^2 d\tau\leq Q(M).
\label{est35}
\end{align}

Thanks to \eqref{est32}--\eqref{est35} we can now apply the uniform Gronwall's lemma to \eqref{diffineq4bis}
and obtain
\begin{align}
&\Phi(t)\leq Q(M),\quad\forall t\geq 1.
\label{est36}
\end{align}
On the other hand, the definition of the phase space $\mathcal{W}_M$ and \eqref{F2} yield
\begin{align}
&\Vert\varphi(t)\Vert_{L^\rho(\Omega)}\leq Q(M),\qquad \Vert\psi(t)\Vert\leq Q(M),\quad \forall t\geq 0.
\label{est41}
\end{align}
Hence, we deduce from \eqref{est36} and \eqref{est41} that
\begin{align}
&\Vert\varphi(t)\Vert_{H^2(\Omega)}\leq Q(M),\quad\forall t\geq 1.\label{est40}
\end{align}

Moreover, \eqref{est36} and \eqref{est40} entail
\begin{align*}
&\Vert\mu(t)\Vert_V\leq Q(M),\quad\forall t\geq 1.
\end{align*}
Also, using \eqref{est40} once more, we have
\begin{align*}
\Vert\nabla F'(\varphi(t))\Vert\leq\Vert F''(\varphi(t))\nabla\varphi(t)\Vert
\leq Q(M),\quad\forall t\geq 1.
\end{align*}
The last two bounds, \eqref{eq2} and elliptic regularity imply
\color{black}
\begin{align}
&\Vert\varphi(t)\Vert_{H^3(\Omega)}\leq Q(M),\quad\forall t\geq 1.\label{est42}
\end{align}

Finally, from \eqref{est36} and $\eqref{est41}_2$, we get
\begin{align}
&\Vert\psi(t)\Vert_V\leq Q(M),\quad\forall t\geq 1.\label{est43}
\end{align}
Thanks to \eqref{est42} and \eqref{est43}, we have thus proven that there exists $\Lambda=\Lambda(M)>0$ such that
\begin{align}
\mathcal{B}_M:=\big\{w:=[\varphi,\psi]\in H^3(\Omega)\times H^1(\Omega):\:\:\Vert\varphi\Vert_{H^3(\Omega)}\leq \Lambda,
\:\:\Vert\psi\Vert_{H^1(\Omega)}\leq\Lambda,\:\:
\mathcal{E}(w)\leq M\big\}\nonumber
\end{align}
is an absorbing set for the semigroup $\{S_M(t)\}_{t\geq 0}$ in $\mathcal{W}_M$. Since $\mathcal{B}_M$ is also compact in $\mathcal{W}_M$, the conclusion follows from \cite[Thm. 2]{PZ}.
\end{proof}

\color{black}
\textbf{Acknowledgments}. {\color{black} The authors thank the reviewers for their careful reading of the manuscript as well as for their remarks which
have helped us to improve the quality and the clarity of this contribution.} The authors are members of the Gruppo Nazionale per l'Analisi Matematica, la Probabilit\`{a} e le loro Applicazioni (GNAMPA) of the Istituto Nazionale di Alta Matematica (INdAM).
\color{black}

\end{document}